\documentclass[11pt,twosided,a4paper]{amsart}
\usepackage{amsmath, amsthm}
\usepackage{amsfonts,amssymb}
\usepackage{graphicx}
\usepackage{calrsfs}

\title[Maximal surfaces in AdS 3-manifolds with particles]{Maximal surfaces in anti-de Sitter 3-manifolds with particles}
\author{J\'{e}r\'{e}my Toulisse}
\address{
Department of Mathematics \\
Mathematics Research Unit BLG \\
University of Southern Califonia \\
3620 S. Vermont Avenue, KAP 104 \\
Los Angeles, CA 90089-2532}
\email{toulisse@usc.edu}
\date{\today}

\newcommand{\isom}{\text{Isom}_+}

\newcommand{\T}{\mathcal{T}_\theta(\Sigma_{\mathfrak{p}})}
\newcommand{\fkp}{\mathfrak{p}}
\newcommand{\Mess}{\mathfrak{M}}
\newcommand{\M}{\mathcal{M}_\theta(\Sigma_{\mathfrak{p}})}
\newcommand{\HH}{\mathcal{H}_\theta(\Sigma_{\mathfrak{p}})}
\newcommand{\I}{\textrm{I}}
\newcommand{\II}{\textrm{II}}
\newcommand{\III}{\textrm{III}}
\newcommand{\W}{\mathcal{W}}
\newcommand{\Z}{\mathcal{Z}}

\newcommand{\N}{\mathcal{N}}
\newcommand{\AdS}{\text{AdS}}
\newcommand{\PSL}{\text{PSL}_2}

\theoremstyle{definition}
\newtheorem{Def}{Definition}[section]
\theoremstyle{plain}
\newtheorem{theo}[Def]{Theorem}
\newtheorem{prop}[Def]{Proposition}

\newtheorem{lemma}[Def]{Lemma}
\theoremstyle{remark}
\newtheorem{rem}{Remark}[section]

\begin{document}

\maketitle

\begin{abstract}
We prove the existence of a unique maximal surface in each anti-de Sitter (AdS) Globally Hyperbolic Maximal (GHM) manifold with particles (that is, with conical singularities along time-like lines) for cone angles less than $\pi$. We interpret this result in terms of Teichm\"uller theory, and prove the existence of a unique minimal Lagrangian diffeomorphism isotopic to the identity between two hyperbolic surfaces with cone singularities when the cone angles are the same for both surfaces and are less than $\pi$.
\end{abstract}

\tableofcontents

\section{Introduction}
For $\theta \in (0,2 \pi)$,  consider the space obtained by gluing with a rotation the boundary of an angular sector of angle $\theta$ between two half-lines in the hyperbolic disk. We denote this singular Riemannian manifold by $\mathbb{H}^2_\theta$. The induced metric is called \textbf{local model for hyperbolic metric with conical singularity of angle $\theta$}. This metric is hyperbolic outside the singular point.

Let $\Sigma_\fkp$ be a closed oriented surface of genus $g$ with $n$ marked points $\fkp:=(p_1,...,p_n)\subset\Sigma$ and $\theta :=(\theta_1,...,\theta_n)\in (0,2\pi)^n$.

\begin{Def}
A hyperbolic metric $g$ with conical singularities of angle $\theta_i$ at the $p_i\in \fkp$ is a (singular) metric on $\Sigma_{\fkp}$ such that each $p_i\in\fkp$ has a neighborhood isometric to a neighborhood of the singular point in $\mathbb{H}^2_{\theta_i}$ and $(\Sigma_\fkp,g)$ has constant curvature $-1$ outside the marked points. 
\end{Def} 

It has been proved by M. Troyanov \cite{troyanov} and M.C. McOwen \cite{mcowen} that each conformal class of metric on a surface $\Sigma_\fkp$ with marked points admits a unique hyperbolic metric with cone singularities of angle $\theta_i$ at the $p_i$ as soon as 
$$\chi(\Sigma_\fkp)+\sum_{i=1}^n\left(\frac{\theta_i}{2\pi}-1\right)<0,$$
 where $\chi(\Sigma_\fkp)$ is the Euler characteristic of $\Sigma_\fkp$.

We denote by $\T$ the space of isotopy classes of hyperbolic metrics with cone singularities of angle $\theta$ (where the isotopies fix each marked point). Note that, from the theorem of Troyanov and McOwen, this space is canonically identified with the space of marked conformal structures on $\Sigma_\fkp$. As in dimension 2, a conformal structure is equivalent to a complex structure, $\T$ is also identified with the space of marked complex structures on $\Sigma_\fkp$.

When $\fkp=\emptyset$, $\T$ corresponds to the classical Teichm\"uller space $\mathcal{T}(\Sigma)$ of $\Sigma$, that is, the space of equivalence classes of marked hyperbolic structures on $\Sigma$.

\medskip
\noindent\textbf{Minimal Lagrangian diffeomorphism.}
\begin{Def}
Let $g_1,g_2 \in \mathcal{T}(\Sigma)$, a minimal Lagrangian diffeomorphism $\Psi: (\Sigma,g_1) \longrightarrow (\Sigma,g_2)$ is an area preserving diffeomorphism such that its graph is a minimal surface in $(\Sigma\times \Sigma, g_1\oplus g_2)$.
\end{Def}

In \cite{MR1201611}, R. Schoen proved the existence of a unique minimal Lagrangian diffeomorphism isotopic to the identity between two hyperbolic surfaces $(\Sigma,g_1)$ and $(\Sigma,g_2)$ (see also \cite{labourie1992surfaces}).

Minimal Lagrangian diffeomorphisms are related to harmonic diffeomorphisms (that is to diffeomorphisms whose differential minimizes the $L^2$ norm). For a conformal structure $\frak{c}$ on $\Sigma$ and $g\in\mathcal{T}(\Sigma)$, the work of  J.J. Eells and J.H. Sampson \cite{eells-sampson} implies the existence of a unique harmonic diffeomorphism $u: (\Sigma,\frak{c}) \to (\Sigma,g)$ isotopic to the identity. Given a harmonic diffeomorphism $u$ we define its \textbf{Hopf differential} by  $\Phi(u):=u^*g^{2,0}$ (that is the $(2,0)$ part with respect to the complex structure associated to $\frak{c}$ of $u^*g$). The work of R. Schoen implies that, given $g_1,g_2\in \mathcal{T}(\Sigma)$, there exists a unique conformal structure $\frak{c}$ on $\Sigma$ such that $\Phi(u_1)+\Phi(u_2)=0$, where $u_i: (\Sigma,\frak{c})\to (\Sigma,g_i)$ is the unique harmonic diffeomorphism isotopic to the identity. Moreover, $u_2\circ u_1^{-1}$ is the unique minimal Lagrangian diffeomorphism isotopic to the identity between $(\Sigma,g_1)$ and $(\Sigma,g_2)$.

\medskip

In his thesis, J. Gell-Redman \cite{gell2010harmonic} proved the existence of a unique harmonic diffeomorphism isotopic to the identity from a closed surface with $n$ marked points equipped with a conformal structure to a negatively curved surface with $n$ conical singularities of angles less than $\pi$ at the marked points (where the isotopy fixes each marked point). In this paper, we prove the existence of a unique minimal Lagrangian diffeomorphism isotopic to the identity between hyperbolic surfaces with conical singularities of angles less than $\pi$ and so we give a positive answer to \cite[Question 6.3]{barbot2012some}.

\begin{theo}\label{min}
Given two hyperbolic metrics $g_1,g_2 \in \T$ with cone singularities of angles $\theta=(\theta_1,...,\theta_n)\in (0,\pi)^n$, there exists a unique minimal Lagrangian diffeomorphism $\Psi: (\Sigma,g_1) \longrightarrow (\Sigma,g_2)$ isotopic to the identity.
\end{theo}

In particular, this result extends the result of R. Schoen to the case of surfaces with conical singularities of angles less than $\pi$. The proof of this statement uses the deep connections between hyperbolic surfaces and three dimensional anti-de Sitter (AdS) geometry.

\medskip
\noindent\textbf{AdS geometry.} An anti-de Sitter (AdS) manifold $M$ is a Lorentz manifold of constant sectional curvature $-1$. It is Globally Hyperbolic Maximal (GHM) when it contains a closed Cauchy surface, that is a space-like surface intersecting every inextensible time-like curve exactly once, and which is maximal in a certain sense (precised in Section \ref{adsmfd}). The global hyperbolicity condition implies in particular that $M$ is homeomorphic to $\Sigma\times \mathbb{R}$ (where $\Sigma$ has the same topology as the Cauchy surface). In his groundbreaking work, G. Mess \cite[Section 7]{mess2007lorentz} considered the moduli space $\mathcal{M}(\Sigma)$ of AdS GHM structure on $\Sigma\times \mathbb{R}$. He proved that $\mathcal{M}(\Sigma)$ is naturally parametrized by two copies of the Teichm\"uller space $\mathcal{T}(\Sigma)$. This result can be thought as an AdS analogue of the famous Bers' simultaneous uniformization  Theorem \cite{bers1960simultaneous}. In fact, Bers' Theorem provides a parametrization of the moduli space $\mathcal{QF}(\Sigma)$ of quasi-Fuchsian structures on $\Sigma\times \mathbb{R}$ by two copies of the Teichm\"uller space $\mathcal{T}(\Sigma)$.

In \cite{barbot2007constant}, the authors proved the existence of a unique maximal space-like surface (that is an area-maximizing surface whose induced metric is Riemannian) in each AdS GHM metric on $\Sigma \times \mathbb{R}$. Note that maximal surfaces are the Lorentzian analogue of minimal surfaces in Riemannian geometry: they are characterized by the vanishing of the mean curvature field. This result is actually equivalent to the result of R. Schoen of existence of a unique minimal Lagrangian diffeomorphism (see \cite{aiyama}).
 
\medskip 
 
  A particle in an AdS GHM manifold $M$ is a conical singularity along a time-like line. In this paper, we only consider particles with cone angles less than $\pi$. In \cite {bonsante2009ads}, F. Bonsante and J.-M. Schlenker extended Mess' parametrization to the case of  AdS GHM manifolds with particles: they gave a parametrization of the moduli space of AdS convex GHM manifolds with particles by two copies of the Teichm\"uller space $\T$. In this paper, we study the existence and uniqueness of a maximal surface in AdS GHM manifolds with particles, and give a positive answer to  \cite[Question 6.2]{barbot2012some}. Namely, we prove
  
\begin{theo}\label{surf}
For each AdS convex GHM 3-manifold $(M,g)$ with particles of angles less than $\pi$, there exists a unique maximal space-like surface $S\hookrightarrow (M,g)$.
\end{theo}

Moreover, we prove that the existence of a unique maximal surface provides the existence of a unique minimal Lagrangian diffeomorphism isotopic to the identity 
$$\Psi: (\Sigma_\fkp,g_1) \longrightarrow (\Sigma_\fkp,g_2),$$
where $g_1,g_2\in\T$ parametrize the AdS convex GHM metric with particles $g$.

\medskip

It follows from Theorem \ref{surf} that one can associate to each pair of hyperbolic metrics with conical singularities $g_1,g_2\in\T$ the first and second fundamental form of the unique maximal surface in $(M,g)$ where $g$ is parametrized by $g_1$ and $g_2$. It gives a map 
$$\varphi: \T\times \T\longrightarrow T^*\T.$$
From Theorem \ref{surf} and using the Fundamental Theorem of surfaces in AdS manifolds with particles (see Section \ref{adsparticles}), we prove that this map is one-to-one. In Theorem \ref{interpretation}, we give a nice geometric interpretation of $\varphi$: given a pair of points $g_1,g_2\in\T$, there exists a unique conformal structure $\frak{c}$ on $\Sigma_\fkp$ such that $\Phi(u_1)+\Phi(u_2)=0$ and $u_2\circ u_1^{-1}$, where $\Phi(u_i)$ is the Hopf differential of the harmonic map $u_i: (\Sigma_\fkp,\frak{c})\to (\Sigma_\fkp,g_i)$. We then have $\varphi(g_1,g_2)=\big( \frak{c},i\Phi(u_1)\big)$. This picture extends the connections between minimal Lagrangian diffeomorphisms and harmonic maps to the case with conical singularities.

Finally, in \cite{toulisse}, we prove the existence of a minimal map between hyperbolic surfaces with conical singularities when the two surfaces have different cone angles. In that case, uniqueness only holds when the cone angles of one surface are strictly smaller than the ones of the other surface.
\medskip
 
\noindent\textbf{Acknowledgement.} It is a pleasure to thank Jean-Marc Schlenker for its patience while discussing about the paper. I would also thank Francesco Bonsante and Thierry Barbot for helpful and interesting conversations about this subject. I am grateful to the referee who helped to improve the paper.

\section{AdS GHM 3-manifolds}\label{adsmfd}

\subsection{Mess parametrization}

\

\medskip

\noindent\textbf{The AdS 3-space.}
Let $\mathbb{R}^{2,2}$ be the usual real 4-space with the quadratic form:
$$q(x):=x_1^2+x_2^2-x_3^2-x_4^2.$$
The anti-de Sitter (AdS) 3-space is defined by:
$$\AdS_3=\{x\in \mathbb{R}^{2,2} \text{ such that } q(x)=-1 \}.$$
With the induced metric, $\AdS_3$ is a Lorentzian symmetric space of dimension 3 with constant curvature $-1$ diffeomorphic to $\mathbb{D}\times S^1$ (where $\mathbb{D}$ is a disk of dimension 2). In particular, $\AdS_3$ is not simply connected.

The Klein model of the AdS 3-space is given by the image of $\AdS_3$ under the canonical projection 
$$\pi : \mathbb{R}^{2,2}\setminus\{0\} \longrightarrow \mathbb{RP}^3.$$
Denote by $\AdS^3:=\pi(\AdS_3)$. In the affine chart $x_4\neq 0$ of $\mathbb{RP}^3$, $\AdS^3$ is the interior of the hyperboloid of one sheet given by the equation $\{x_1^2+x_2^2-x_3^2=1\}$, and this hyperboloid identifies with the boundary $\partial \AdS^3$ of $\AdS^3$ in this chart. In this model, geodesics are given by straight lines: space-like geodesics are the ones which intersect the boundary $\partial \AdS^3$ in two points, time-like geodesics are the ones which do not have any intersection and light-like geodesics are tangent to $\partial \AdS^3$.
 
 \begin{rem}
This model is called Klein model by analogy with the Klein model of the hyperbolic space. In fact, in both models, geodesics are given by straight lines.
 \end{rem}

\medskip
\noindent\textbf{The isometry group.} 
As $\partial \AdS^3$ is a hyperboloid of one sheet, it is foliated by two families of straight lines. We call one family the right one and the other, the left one. The group $\isom(\AdS^3)$ of space and time-orientation preserving isometries of $\AdS^3$ preserves each family of the foliation. Fix a space-like plane $P_0$ in $\AdS^3$, its boundary is a space-like circle in $\partial \AdS^3$ which intersects each line of the right (respectively the left) family exactly once. Then $P_0$ provides an identification of each family with $\mathbb{RP}^1$ (when changing $P_0$ to another space-like plane, the identification changes by a conjugation by an element of $\PSL(\mathbb{R})$). It is proved in \cite[Section 7]{mess2007lorentz} that each element of $\isom(\AdS^3)$ acts by projective transformations on each $\mathbb{RP}^1$ and so extend to a pair of elements in $\PSL(\mathbb{R})$. So $\isom(\AdS^3)\cong \PSL(\mathbb{R})\times \PSL(\mathbb{R})$.

\begin{rem}\label{boundary} Fixing a space-like plane $P_0$ also provides an identification between $\partial \AdS^3$ and $\mathbb{RP}^1\times \mathbb{RP}^1$. In fact, given a point $x\in \partial \AdS^3$, there exists a unique line in the right family a unique line the left one which pass through $x$. It follows that $x\in \partial \AdS^3$ gives a point in $\mathbb{RP}^1\times\mathbb{RP}^1$. This application is bijective.
\end{rem}

\medskip
\noindent\textbf{AdS GHM 3-manifold.}
An AdS 3-manifold is a manifold $M$ endowed with a $(G,X)$-structure, where $G=\isom(\AdS^3)$, $X=\AdS^3$. That is, $M$ is endowed with an atlas of charts taking values in $\AdS^3$ so that the transition functions are restriction of elements in $\isom(\AdS^3)$. An AdS 3-manifold $M$ is Globally Hyperbolic Maximal (GHM) if it satisfies the following two conditions:

\begin{enumerate}
\item\textbf{Global Hyperbolicity:} $M$ contains a space-like Cauchy surface, that is a closed oriented surface which intersects every inextensible time-like curve exactly once.
\item\textbf{Maximality:} $M$ cannot be strictly embedded in an AdS manifold satisfying the same properties.
\end{enumerate}

Note that the Global Hyperbolicity condition implies strong restrictions on the topology of $M$. In particular, $M$ has to be homeomorphic to $\Sigma\times \mathbb{R}$ where $\Sigma$ is an oriented closed surface of genus $g>0$ (homeomorphic to the Cauchy surface). We restrict ourselves to the case $g>1$. We denote by $\mathcal{M}(\Sigma)$ the space of AdS GHM structure on $\Sigma \times \mathbb{R}$ considered up to isotopy, and by $\mathcal{T}(\Sigma)$ the Teichm\"uller space of $\Sigma$.

We have a fundamental result due to G. Mess \cite[Proposition 20]{mess2007lorentz}:

\begin{theo}[Mess]
There is a parametrization $\Mess:  \mathcal{M}(\Sigma) \longrightarrow \mathcal{T}(\Sigma) \times \mathcal{T}(\Sigma)$.
\end{theo}

\begin{proof}[Construction of the parametrization]
To an AdS GHM structure on $M$ is associated its holonomy representation $\rho :\pi_1(M) \to \isom(\AdS^3)$ (well defined up to conjugation). As $\isom(\AdS^3)\cong \PSL(\mathbb{R})\times \PSL(\mathbb{R})$ and as $\pi_1(M)=\pi_1(\Sigma)$, one can split the representation $\rho$ into two morphisms 
$$\rho_1,\rho_2: \pi_1(\Sigma) \to \PSL(\mathbb{R}).$$ G. Mess proved \cite[Proposition 19]{mess2007lorentz} that these holonomies have maximal Euler class $e$ (that is $\vert e(\rho_l)\vert=\vert e(\rho_r)\vert=2g-2$). Using Goldman's criterion \cite{MR952283}, he proved that these morphisms are Fuchsian holonomies and so define a pair of points in $\mathcal{T}(\Sigma)$.

Reciprocally, as two Fuchsian holonomies $\rho_1, \rho_2$ are conjugated by an orientation preserving homeomorphism $\phi: \mathbb{RP}^1\to \mathbb{RP}^1$ and as $\partial \AdS^3$ identifies with $\mathbb{RP}^1\times \mathbb{RP}^1$ (fixing a totally geodesic space-like plane $P_0$, see Remark \ref{boundary}), one can see the graph of $\phi$ as a closed curve in $\partial \AdS^3$. G. Mess proved that this curve is nowhere time-like and is contained in an affine chart. In particular, one can construct the convex hull of the graph of $\phi$. The holonomy $(\rho_1,\rho_2): \pi_1(\Sigma)\to \isom(\AdS^3)$ acts properly discontinuously on this convex hull and the quotient is a piece of globally hyperbolic AdS manifold. It follows from a Theorem of Y. Choquet-Bruhat and R. Geroch \cite{choquet} that this piece of  AdS globally hyperbolic manifold uniquely embeds in a maximal one. So the map $\Mess$ is a one-to-one.
\end{proof}

\subsection{Surfaces embedded in an AdS GHM 3-manifold}

K. Krasnov and J.-M. Schlenker  \cite[Section 3]{krasnov2007minimal} proved results about surfaces embedded in an AdS GHM manifold. Here we state some of these results. Recall that a space-like surface embedded in a Lorentzian manifold is maximal if its mean curvature vanishes everywhere. The following result was proved by T. Barbot, F. B\'eguin and A. Zeghib in \cite{barbot2007constant}:

\begin{theo}[Barbot, B\'eguin, Zeghib]
Every AdS GHM 3-manifold contains a unique maximal space-like surface.
\end{theo}

In \cite{krasnov2007minimal}, the authors give an explicit formula for the Mess parametrization $\Mess$:

\begin{theo}[Krasnov, Schlenker]
Let $S$ be a space-like surface embedded in an AdS GHM manifold $M$ whose principal curvatures are in $(-1,1)$. We denote by E the identity map, $J$ the complex structure on S (associated to the induced metric), $B$ its shape operator and $\I$ its first fundamental form. We have:
$$\Mess(M)=(g_1,g_2),$$
where $g_{1,2}(x,y)=\I((E\pm JB)x,(E\pm JB)y)$.
\end{theo}

\begin{rem}
In particular, they proved that the metrics $g_1$ and $g_2$ are hyperbolic and do not depend of the choice of the surface $S$ (up to isotopy).
\end{rem}

If we denote by $\mathcal{H}(\Sigma)$ the space of maximal space-like surfaces in germs of AdS manifold, it is proved in \cite{krasnov2007minimal} (using the Fundamental Theorem of surfaces embedded in AdS manifolds) that this space is canonically identified with the space of couples $(g,h)$ where $g$ is a smooth metric on $\Sigma$ and $h$ is a symmetric bilinear form on $TS$ so that:
\begin{enumerate}
\item $tr_g(h)=0$.
\item $\delta_g h=0$ (where $\delta_g$ is the divergence operator associated to the Levi-Civita connection of $g$).
\item $K_g = -1-\det_g(h)$ (where $K_g$ is the Gauss curvature). We call this equation \textbf{modified Gauss' equation}.
\end{enumerate}

We recall a theorem of Hopf \cite{hopf1950flachen}:

\begin{theo}[Hopf]
Let $g$ be a Riemannian metric on $\Sigma$ and $h$ a bilinear symmetric form on $T\Sigma$, then:
\begin{itemize}
\item[i.] $tr_g(h)=0$ if and only if $h$ is the real part of a quadratic differential $q$ on $(\Sigma,g)$
\item[ii.] If i. holds, then $\delta_g h=0$ if and only if $q$ is holomorphic with respect to the complex structure associated to $g$.
\item[iii.] if i. and ii. hold, then $g$ (respectively $h$) is the first (respectively second) fundamental form of a maximal surface if and only if $K_g = -1-\det_g(h)$.
\end{itemize}
\end{theo}

Moreover, it is proved in \cite[Lemma 3.6.]{krasnov2007minimal} that for every conformal class $\frak{c}$ on $\Sigma$ and every $h$ real part of a holomorphic quadratic differential $q$ on $(\Sigma,J_\frak{c})$ (where $J_\frak{c}$ is the complex structure associated to $\frak{c}$), there exists a unique metric $\mathrm{g}_0\in\frak{c}$ such that modified Gauss' equation is satisfied. 

This result provides a canonical parametrization of $\mathcal{H}(\Sigma)$ by $T^*\mathcal{T}(\Sigma)$. In this parametrization, $h$ is the real part of a holomorphic quadratic differential, and $\mathrm{g}_0\in\frak{c}$ is the unique metric verifying $K_{\mathrm{g}_0}=-1-\det_{\mathrm{g}_0}(h)$. In addition, such a surface has principal curvatures in $(-1,1)$ \cite[Lemma 3.11.]{krasnov2007minimal}.

As every AdS GHM manifold contains a unique maximal surface, there is a parametrization $\phi: \mathcal{M}(\Sigma) \longrightarrow T^*\mathcal{T}(\Sigma)$ \cite[Theorem 3.8]{krasnov2007minimal}. Hence, we get an application associated to the Mess parametrization: $$\varphi:=\phi\circ\Mess^{-1} : T^*\mathcal{T}(\Sigma) \to \mathcal{T}(\Sigma)\times\mathcal{T}(\Sigma).$$

\section{AdS convex GHM 3-manifolds with particles}\label{adsparticles}

In this section we define the AdS convex GHM manifolds with particles and recall the parametrization of the moduli space of such structures. The proofs of these results can be found in \cite{krasnov2007minimal} and \cite{bonsante2009ads}.

\subsection{Extension of Mess' parametrization}\label{extensionmessparam}

First, we are going to define the singular AdS space of dimension 3 in order to define the AdS convex GHM manifolds with particles.

\begin{Def}
Let $\theta>0$, we define $\AdS^3_\theta:=\{(t,\rho,\varphi)\in \mathbb{R}\times\mathbb{R}_{\geq 0}\times[0,\theta)\}$ with the metric:
$$g_\theta=-\cosh^2\rho dt^2 + d\rho^2+\sinh^2\rho d\varphi^2.$$
\end{Def}

\begin{rem}
\begin{itemize}
\item $\AdS^3_\theta$ can be obtained by cutting the universal cover of $\AdS^3$ along two time-like planes intersecting along the line $l:=\{\rho=0\}$, making an angle $\theta$, and gluing the two sides of the angular sector of angle $\theta$ by a rotation fixing $l$. A simple computation shows that, outside of the singular line, $\AdS^3_\theta$ is a Lorentz manifold of constant curvature -1, and $\AdS^3_\theta$ carries a conical singularity of angle $\theta$ along $l$.
\item In the neighborhood of the totally geodesic plane $P_0:= \{t=0\}$ given by the points at a causal distance less than $\pi/2$ from $P_0$, the metric $g_\theta$ also expresses
$$g_\theta=-dt^2 + \cos^2t(d\rho^2+\sinh^2\rho d\varphi^2).$$
\end{itemize}
\end{rem}

\begin{Def}
An AdS cone-manifold is a (singular) Lorentzian 3-manifold $(M,g)$ in which any point $x$ has a neighborhood isometric to an open subset of $\AdS^3_\theta$ for some $\theta>0$. If $\theta$ can be taken equal to $2\pi$, $x$ is a smooth point, otherwise $\theta$ is uniquely determined.
\end{Def}

To define the global hyperbolicity in the singular case, we need to define the orthogonality to the singular locus:

\begin{Def}
Let $S\subset \AdS^3_\theta$ be a space-like surface which intersect the singular line $l$ at a point $x$. $S$ is said to be orthogonal to $l$ at $x$ if the causal distance (that is the ``distance'' along a time-like line) to the totally geodesic plane $P$ orthogonal to the singular line at $x$ is such that:
$$\lim\limits_{y\to x,y\in S}\frac{d(y,P)}{d_S(x,y)}=0$$
where $d_S(x,y)$ is the distance between $x$ and $y$ along $S$.

Now, a space-like surface $S$ in an AdS cone-manifold $(M,g)$ which intersects a singular line $d$ at a point $y$ is said to be orthogonal to $d$ if there exists a neighborhood $U\subset M$ of $y$ isometric to a neighborhood of a singular point in $\AdS^3_\theta$ such that the isometry sends $S\cap U$ to a surface orthogonal to $l$ in $\AdS^3_\theta$.
\end{Def}

Now we are able to define the AdS convex GHM manifolds with particles.

\begin{Def} An  AdS convex GHM manifold with particles is an AdS cone-manifold $(M,g)$ which is homeomorphic to $\Sigma_\fkp\times \mathbb{R}$ (where $\Sigma_\fkp$ is a closed oriented surface with $n$ marked points), such that the singularities are along time-like lines $d_1,...,d_n$ and have fixed cone angles $\theta_1,..,\theta_n$ with $\theta_i<\pi$. Moreover, we impose two conditions:

\begin{enumerate}
\item\textbf{Convex Global Hyperbolicity} $M$ contains a space-like future-convex Cauchy surface orthogonal to the singular locus.
\item\textbf{Maximality} $M$ cannot be strictly embedded in another manifold satisfying the same conditions.
\end{enumerate}
\end{Def}

\begin{rem}
The condition of convexity in the definition will allow us to use a convex core. As pointed out by the authors in \cite{bonsante2009ads}, we do not know if every AdS GHM manifold with particles is convex GHM.
\end{rem}

\begin{Def}
For $\theta:=(\theta_1,...,\theta_n)\in (0,\pi)^n$, let $\M$ be the space of isotopy classes of AdS convex GHM metrics on $M=\Sigma_\fkp \times \mathbb{R}$ with particles of cone angles $\theta_i$ along $d_i$. 
\end{Def}

Many results known in the non-singular case extend to the singular case (that is with particles of angles less than $\pi$). We recall some of them here (see \cite{bonsante2009ads}, \cite{krasnov2007minimal}):
\begin{enumerate}
\item The parametrization $\Mess$ defined above extends to the singular case. Namely, we have a parametrization $\Mess_\theta: \M\longrightarrow \T\times \T$ which corresponds to Mess' parametrization when there is no particle.
\item Each AdS convex GHM 3-manifold with particles $(M,g)$ contains a minimal non-empty convex subset called its ''convex core'' whose boundary is a disjoint union of two pleated space-like surfaces orthogonal to the singular locus (except in the Fuchsian case which corresponds to the case where the two metrics of the parametrization are equal. In this case, the convex core is a totally geodesic space-like surface).

\end{enumerate}

\begin{rem}
The analogy between AdS GHM geometry and quasi-Fuchsian geometry explained in the introduction extends to the case with particles. Namely, it is proved in \cite{conebend} and \cite{qfmp} that there exists a parametrization of the moduli space of quasi-Fuchsian manifolds with particles which extends Bers' parametrization. 
\end{rem}

\subsection{Maximal surface}

Let $g\in \M$ be an AdS convex GHM metric with particles on $M=\Sigma_\fkp\times \mathbb{R}$.

\begin{Def}
A maximal surface in $(M,g)$ is a locally area-maximizing space-like Cauchy surface $S\hookrightarrow (M,g)$ which is orthogonal to the singular lines.
\end{Def}

In particular, such a maximal surface $S\hookrightarrow (M,g)$ has everywhere vanishing mean curvature. Note that our definition differs from  \cite[Definition 5.6]{krasnov2007minimal} where the authors impose the boundedness of the principal curvatures of $S$. The following Proposition shows that a maximal surface in our sense has bounded principal curvatures:

\begin{prop}\label{simplepoles}
For a maximal surface $S\hookrightarrow (M,g)$ with shape operator $B$ and induced metric $g_S$, $\det_{g_S}(B)$ tends to zero at the intersections with the particles. In particular, $B$ is the real part of a meromorphic quadratic differential with at most simple poles at the singularities.
\end{prop}

\begin{proof}
Let $d$ be a particle of angle $\theta$ and set $0:=d\cap S$. We see locally $S$ as the graph of a function $u: P_0\longrightarrow \mathbb{R}$ where $P_0$ is the (piece of) totally geodesic plane orthogonal to $d$ at $0$. We will show that, the induced metric $g_S$ on $S$ carries a conical singularity of angle $\theta$.

Recall that a metric $g$ on a surface carries a conical singularity of angle $\theta$ if there exists complex coordinates $z$ centered at the singularity so that
$$g= e^{2u}\vert z\vert^{2\left(\theta/2\pi -1\right)}\vert dz\vert^2,$$
where $u$ is a bounded function. We need the following lemma:

\begin{lemma}
The gradient of $u$ tends to zero at the intersections with the particles.
\end{lemma}
\begin{proof}
To prove this lemma, we will use Schauder estimates for solutions of uniformly elliptic PDE's. For the convenience of the reader, we recall these estimates. The main reference for the theory is \cite{gilbarg}.

A second order linear operator $L$ on a domain $\Omega\subset \mathbb{R}^n$ is a differential operator of the form
$$Lu= a^{ij}(x)D_{ij}u+ b^k(x)D_k u + c(x)u,~u\in \mathcal{C}^2(\Omega),~x\in \Omega,$$
where we sum over all repeated indices. We say that $L$ is uniformly elliptic if the smallest eigenvalue of the matrix $\big(a_{ij}(x)\big)$ is bounded from below by a strictly positive constant.

We finally define the following norms for a function $u$ on $\Omega$:
\begin{itemize}
\item $\vert u \vert_k:= \Vert u \Vert_{\mathcal{C}^k(\Omega)}.$
\item $\vert u \vert^{(i)}_0 := \underset{x\in\Omega}{\sup}~ d_x^i \vert u(x)\vert,~\text{where }d_x=\text{dist}(x,\partial\Omega).$
\item $ \vert u \vert^*_k= \sum_{i=0}^k \underset{x\in \Omega,~\vert\alpha\vert=i}{\sup} d_x^i\vert D^\alpha u\vert$.
\end{itemize}
The following theorem can be found in \cite[Theorem 6.2]{gilbarg}
\begin{theo}\label{schauder}(Schauder interior estimates)
Let $\Omega\subset \mathbb{R}^n$ be a domain with $\mathcal{C}^2$ boundary and $u\in \mathcal{C}^2(\Omega)$ be solution of the equation
$$L u=0$$
where $L$ is uniformly elliptic so that
$$\left\vert a^{ij}\right\vert_0^{(0)},~\left\vert b^k\right\vert^{(1)}_0,~\vert c\vert^{(2)}_0 < \Lambda.$$
Then there exists a positive constant $C$ depending only on $\Omega$ and $L$ so that
$$\vert u\vert^*_2 \leq C \vert u\vert_0.$$
\end{theo}

For every domain $\Omega\subset P_0$ which does not contain the singular point, $u$ satisfies the maximal surface equation (see for example \cite{gerhardt1983h}) which is given by:
$$\mathcal{L}(u):=\text{div}_{g_S}\big(v(-1,\pi^*\nabla u)\big)=0.$$
Here, $\pi: S \longrightarrow P_0$ is the orthogonal projection, $v=\big( 1-\Vert \pi^*\nabla u \Vert^2\big)^{-1/2}$ and so $v(-1,\pi^*\nabla u)$ is the unit future pointing normal vector field to $S$. Also, one easily checks that this equation can be written
\begin{equation}\label{2max}
\text{div}_{g_S}(v \pi^*\nabla u)+ a(x,u,\nabla u)=0, \text{ for some function }a.
\end{equation}
The proof of Proposition \ref{spacelike} applies in this case and implies the $S$ is uniformly space-like. It follows that  $\pi$ is uniformly bi-Lipschitz and so $v$ is uniformly bounded.

It follows that Equation (\ref{2max}) is a quasi-linear elliptic equation in the divergence form. Moreover, if we write it in the following way:
$$a^{ij}(x,u,Du) D_{ij}u + b^k(x,u,Du)D_ku + c(x,Du,u)u=0,$$
it is easy to see that the equation is uniformly elliptic (in fact $a^{ij}(x,u,Du)\geq 1$) and the coefficients satisfy conditions of Theorem \ref{schauder} (as they are uniformly bounded on $\Omega$). Hence, we are in the good framework to apply the Schauder estimates.

Let $x_0\in P_0\setminus\{0\}$ and let $2r:=\text{dist}_S(x_0,0)$. Consider the disk $D_r$ of radius $r$ centered at $x_0$. It follows from the previous discussion that  $u: D_r \longrightarrow \mathbb{R}$ satisfies $\mathcal{L}u=0$. By a homothety of ratio $1/r$, send the disk $D_r$ to the unit disk $(D,h_r)$ where $h_r$ is the metric of constant curvature $-r^2$. The function $u$ is sent to a new function
$$u_r : (D,h_r) \longrightarrow \mathbb{R},$$
and satisfies the equation 
$$\mathcal{L}_r u_r=0.$$
Here, the operator $\mathcal{L}_r$ is the maximal surface operator for the rescaled metric $g_r:= -dt^2+\cos^2t. h_r$. In particular, $\mathcal{L}_r$ is a quasi-linear uniformly elliptic operator whose coefficients applied to $u_r$ satisfy the condition of Theorem \ref{schauder}.

In a polar coordinates system $(\rho,\varphi)$, the metric $h_r$ expresses
$$h_r=d\rho^2+r^{-2}\sinh^2 (r.\rho)d\varphi^2.$$
As $r$ tends to zero, the metric $h_r$ converges $\mathcal{C}^\infty$ on $D$ to the flat metric $h_0=d\rho^2+\rho^2d\varphi^2$. It follows that the coefficients of the family of operators $(\mathcal{L}_r)_{r\in (0,1)}$ applied to $u_r$ converge to the ones of the operator $\mathcal{L}_0$ applied to $u_0=\underset{r\to 0}{\lim} u_r$ where $\mathcal{L}_0$ is the maximal surface operator associated to the metric $g_0=-dt^2+\cos^2t h_0$. 

As a consequence, the family of constants $\{C_r\}$ associated to the Schauder interior estimates applied to $\mathcal{L}_r(u_r)$ are uniformly bounded by some $C>0$.

Now, to obtain a bound on the norm of the gradient $\Vert \nabla u\Vert$ at a point $x_0$ at a distance $2r$ from the singularity, we apply the Schauder interior estimates to $\mathcal{L}_r(u_r)$, where $u_r : (D,h_r) \longrightarrow \mathbb{R}$. We get
$$\vert u_r\vert^*_2 \leq C_r \vert u_r\vert_0 \leq C \vert u_r\vert_0.$$
As $\Vert \nabla u_r \Vert(x_0)\leq \vert u \vert^*_2$, and as $u_r(x_0)= o(2r)$ (because $S$ is orthogonal to $d$), we obtain
$$\Vert \nabla u_r \Vert (x_0) \leq C. o(r).$$
But as $u_r$ is obtained by rescaling $u$ with a factor $r$, so $\Vert \nabla u \Vert = r^{-1} \Vert \nabla u_r \Vert$ and we finally get:
$$\Vert \nabla u\Vert = o(1).$$
\end{proof}
\begin{lemma}
The induced metric $g_S$ on $S$ carries a conical singularity of angle $\theta$ at its intersection with the particle $d$.
\end{lemma}

\begin{proof}
Recall that (see \cite[Section 2.2]{mazzeo1} and \cite[Section 2.1]{mazzeo2}) a metric $h$ carries a conical singularity of angle $\theta$ if and only if there exists normal polar coordinates $(\rho,\varphi)\in \mathbb{R}_{>0}\times [0,2\pi)$ around the singularity so that
$$g=d\rho^2+ f^2(\rho,\varphi)d\varphi^2,~ \frac{f(\rho,\varphi)}{\rho}\underset{\rho\to 0}{\longrightarrow} \theta/2\pi.$$
That is, if $g$ can be written by the matrix
$$g= \left( \begin{array}{ll} 1 & 0 \\ 0 & \left(\frac{\theta}{2\pi}\right)^2\rho^2 + o(\rho^2) \end{array}\right).$$
The metric of $(M,g)$ can be locally written around the intersection of $S$ and the particle $d$ by
$$g=-dt^2 + \cos^2t h_\theta,$$
where $h_\theta=d\rho^2+\left(\frac{\theta}{2\pi}\right)^2\sinh^2\rho d\varphi^2$ is the metric of $\mathbb{H}^2_\theta$.

Setting $t=u(\rho,\varphi)$, with $u(\rho,\varphi)=o(\rho)$ and $\Vert \nabla u\Vert=o(1)$, we get
$$dt^2 = (\partial_\rho u)^2d\rho^2+2\partial_\rho u\partial_\varphi u d\rho d\varphi+ (\partial_\varphi u)^2d\varphi^2.$$
Note that, as $\Vert \nabla u\Vert= o(1)$, $\partial_\rho u = o(1)$ and $\partial_\varphi u = o(\rho)$. 

Finally, using $\cos^2(u)= 1+o(\rho^2)$, we get the following expression for the induced metric on $S$:
$$g_S= \left(\begin{array}{ll} 1 + o(1) & o(\rho) \\ o(\rho) & \left(\frac{\theta}{2\pi}\right)^2\rho^2 + o(\rho^2) \end{array}\right).$$
One easily checks that, with a change of variable, the induced metric carries a conical singularity of angle $\theta$ at the intersection with $d$.
\end{proof}

Now the proof of Proposition \ref{simplepoles} follows: suppose the second fundamental form $\II=g_S(B.,.)$ is the real part of a meromorphic quadratic differential $q$ with a pole of order $n$. In complex coordinates, write $q=f(z)dz^2$ and $g_S= e^{2u}\vert z\vert^{2\left(\theta/2\pi -1\right)}\vert dz\vert^2$ where $u$ is bounded. Then $B$ is the real part of the harmonic Beltrami differential 
$$\mu := \frac{\overline{q}}{g_S}=e^{-2u}\vert z\vert^{-2(\theta/2\pi-1)} \overline{f}(z) d\overline{z}\partial_z.$$
Using the real coordinates $z=x+iy,~dz=dx+idy,~\partial_z= \frac{1}{2}(\partial_x-i\partial_y)$ we get

\begin{eqnarray*}
B & = & \Re\left( \frac{1}{2}e^{-2u}\vert z\vert^{-2(\theta/2\pi-1)}\big(\Re(f)-i\Im(f)\big)(dx-idy)(\partial_x-i\partial_y)\right) \\
& = & \frac{1}{2}e^{-2u}\vert z\vert^{-2(\theta/2\pi-1)}\big(\Re(f)(dx\partial_x-dy\partial_y)- \Im(f)(dx\partial_y-dy\partial_x)\big) \\
& = & \frac{1}{2} e^{-2u}\vert z \vert^{-2(\theta/2\pi-1)}\left(\begin{array}{ll} \Re(f) & -\Im(f) \\ -\Im(f) & -\Re(f) \end{array}\right).
\end{eqnarray*}
It follows that 
$$\text{det}_{g_S}(B)=-\frac{1}{4}e^{-4u}\vert z \vert^{-4(\theta/2\pi-1)} \vert f\vert^2= -e^v \vert z\vert^{-2(\theta/\pi -2 + n)},$$
for some bounded $v$. By (modified) Gauss equation, the curvature $K_S$ of $S$ is given by
$$K_S=-1-\text{det}_{g_S}(B).$$
By Gauss-Bonnet formula for surface with cone singularities (see for example \cite{troyanov}), $K_S$ has to be locally integrable. But we have:
$$K_s dvol_S =\big( -1 + e^w\vert z\vert^{-2(\theta/2\pi-1+n)}\big)d\lambda,$$
where $d\lambda$ is the Lebesgue measure on $\mathbb{R}^2$. It follows that $K_s dvol_S$ is integrable if and only if $\theta/2\pi -1 + n <1$, that is $n\leq 1$. Note also that, for $n\leq 1$, $\det_{g_S}(B)= O\left(\vert z\vert^{2(1-\theta/\pi)}\right)$ and so tends to zero at the singularity.
\end{proof}

It is proved in \cite{krasnov2007minimal} that, as in the non-singular case, we can define the space $\HH$ of maximal surfaces in a germ of AdS convex GHM with $n$ particles of angles $\theta=(\theta_1,...,\theta_n)\in(0,\pi)^n$. This space is still parametrized by $T^*\T$. Recall that the cotangent space $T^*_g\T$ to $\T$ at a metric $g\in \T$ is given by the space of meromorphic quadratic differentials on $(\Sigma_\fkp,J_g)$ (where $J_g$ is the complex structure associated to $g$) with at most simple poles at the marked points.

Moreover, given $(g,h)\in \HH$, using the Fundamental Theorem of surfaces in AdS convex GHM manifolds with particles, one can locally reconstruct a piece of AdS globally hyperbolic manifold with particles which uniquely embeds in a maximal one. It provides a map from $\HH$ to $\M$. This map is bijective if and only if each AdS convex GHM manifold $(M,g)$ contains a unique maximal surface.

\section{Existence of a maximal surface}\label{conical}

In this section, we prove the existence part of Theorem \ref{surf}. Note that in the Fuchsian case (that is when the two metrics of the parametrization $\Mess_\theta$ are equal), the convex core is reduced to a totally geodesic plane orthogonal to the singular locus which is thus maximal (its second fundamental form vanishes). 

Hence, from now on, we consider an AdS convex GHM manifold with particles $(M,g)$, where $g\in \M$ is such that $\Mess_\theta(g)\in \T\times\T$ are two distinct points (that is $(M,g)$ is not Fuchsian). It follows from \cite[Section 5]{bonsante2009ads} that $(M,g)$ contains a convex core with non-empty interior. The boundary of this convex core is given by two pleated surfaces: a future-convex one and a past-convex one.

\begin{prop}\label{existence}
The AdS convex GHM manifold with particles $(M,g)$ contains a maximal surface $S\hookrightarrow (M,g)$.
\end{prop}
The proof is done in four steps:
\begin{enumerate}
\item[\textbf{Step 1}] Approximate the singular metric $g$ by a sequence of smooth metrics $(g_n)_{n\in \mathbb{N}}$ which converges to the metric $g$, and prove the existence for each $n\in \mathbb{N}$ of a maximal surface $S_n\hookrightarrow (M,g_n)$.
\item[\textbf{Step 2}] Prove that the sequence $(S_n)_{n\in \mathbb{N}}$ converges outside the singular lines to a smooth nowhere time-like surface $S$ with vanishing mean curvature. 
\item[\textbf{Step 3}] Prove that the limit surface $S$ is space-like.
\item[\textbf{Step 4}] Prove that the limit surface $S$ is orthogonal to the singular lines.
\end{enumerate}

\subsection{First step}

\

\noindent\textbf{Approximation of singular metrics.}
Take $\theta \in (0,2\pi )$ and let $\mathcal{C}_{\theta} \subset \mathbb{R}^3$ be the cone given by the parametrization:

$$\mathcal{C}_{\theta} := \left\{ \left(u.\cos v,u.\sin v,\text{cotan}(\theta/2).u\right),~ (u,v)\in \mathbb{R}_+\times  [0,2\pi ) \right\}.$$

Now, consider the intersection of this cone with the Klein model of the hyperbolic 3-space, and denote by $h_{\theta}$ the induced metric on $\mathcal{C}_\theta$. Outside the apex, $\mathcal{C}_\theta$ is a convex ruled surface in $\mathbb{H}^3$, and so has constant curvature $-1$. Moreover, one easily checks that $h_\theta$ carries a conical singularity of angle $\theta$ at the apex of $\mathcal{C}_\theta$. Consider the orthogonal projection $p$ from $\mathcal{C}_\theta$ to the disk of equation $\mathbb{D}:=\{z=0\}\subset \mathbb{H}^3$. We have that $(\mathbb{D},(p^{-1})^*h_\theta)$ is isometric to the local model of hyperbolic metric with cone singularity $\mathbb{H}^2_\theta$ as defined in the introduction.

\begin{rem} The angle of the singularity is given by $\displaystyle{\underset{\rho\to 0}{\lim}\frac{l(C_\rho)}{\rho}}$ where $l(C_\rho)$ is the length of the circle of radius $\rho$ centered at the singularity.
\end{rem}

Now, to approximate this metric, take $(\epsilon_n)_{n\in \mathbb{N}}\subset (0,1)$, a sequence decreasing to zero and define a sequence of  even functions $f_n:\mathbb{R}\longrightarrow \mathbb{R}$ so that for each $n\in\mathbb{N}$,
$$\left \{
\begin{array}{l}
f_n (0)  =  -\epsilon_n^2 .\text{cotan}(\theta/2) \\
f_n^{''}(x) < 0 \ \  \forall x \in (-\epsilon_n, \epsilon_n) \\
f_n(x) =  -\text{cotan}(\theta/2).x \text{ if }  x\geqslant\epsilon_n.

\end{array}
\right .$$

\begin{figure}[!h] 
\begin{center}
\includegraphics[height=6cm]{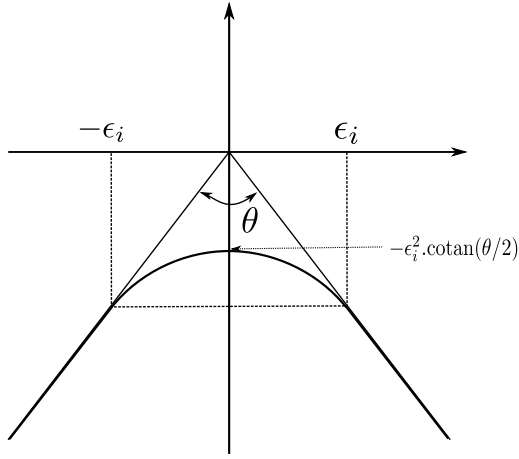}
\end{center}
\caption{Graph of $f_n$} 
\label{fi}
\end{figure}

Consider  the surface $\mathcal{C}_{\theta , n}$ obtained by making a rotation of the graph of $f_n$ around the axis $(0z)$ and consider its intersection with the Klein model of hyperbolic 3-space. Denote by  $h_{\theta,n}$ the induced metric on $\mathcal{C}_{\theta,n}$, and define $\mathbb{H}^2_{\theta,n}:=(\mathbb{D},(p^{-1})^*h_{\theta,n})$ (where $p$ is still the orthogonal projection to the disk $\mathbb{D}=\{z=0\}\subset \mathbb{H}^3$). By an abuse of notations, we write $\mathbb{H}^2_{\theta,n}=(\mathbb{D},h_{\theta,n})$. Denote by $B_i\subset \mathbb{D}$ the smallest set where the metric $h_{\theta,n}$ does not have of constant curvature $-1$, by construction, $B_n \underset{n\to\infty}{\longrightarrow} \{0\}$, where $\{0\}$ is the center of $\mathbb{D}$. We have

\begin{prop} 
For all compact $K\subset \mathbb{D}\setminus\{0\}$, there exists $i_K\in \mathbb{N}$ such that for all $n>n_K$, $h_{\theta_{\vert K}}=h_{\theta ,n_{\vert K}}$.
\end{prop}

We define the AdS 3-space with regularized singularity:

\begin{Def}
For $\theta>0$, $n\in \mathbb{N}$, set $\AdS^3_{\theta,n}:=\{ (t,\rho,\varphi)\in (-\pi/2,\pi/2)\times\mathbb{D}\}$ endowed with the metric:
$$g_{\theta,n}=-dt^2+\cos^2th_{\theta,n}.$$
\end{Def}

By construction, there exists a smallest tubular neighborhood $V^n_\theta$ of $l=\{0\}\times (-\pi/2,\pi/2)$ such that $\AdS^3_{\theta,n}\setminus V^n_\theta$ is a Lorentzian manifold of constant curvature $-1$.

In this way, we are going to define the regularized AdS convex GHM manifold with particles.

For all $j\in\{1,...,n\}$ and $x\in d_j$ where $d_j$ is a singular line in $(M,g)$, there exists a neighborhood of $x$ in $(M,g)$ isometric to a neighborhood of a point on the singular line in $\AdS^3_{\theta_j}$. For $n\in \mathbb{N}$, we define the regularized metric $g_n$ on $M$ so that the neighborhoods of points of $d_j$ are isometric to neighborhoods of points on the central axis in $\AdS^3_{\theta_j,n}$. Clearly, the metric $g_n$ is obtained taking locally the metric of $V^n_{\theta_j}$ in a tubular neighborhood $U^n_j$ of the singular lines $d_j$ for all $j\in\{1,...,n\}$. In particular, outside these $U^n_j$, $(M,g_n)$ is a regular AdS manifold.

\begin{prop}
Let $K\subset M$ be a compact set which does not intersect the singular lines. There exists $n_K\in \mathbb{N}$ such that, for all $n>n_K$, $g_{n_{\vert K}} = g_{\vert K}$.
\end{prop}

\medskip
\noindent\textbf{Existence of a maximal surface in each $(M,g_n)$}

We are going to prove Proposition \ref{existence} by convergence of maximal surfaces in each $(M,g_n)$. A result of Gerhardt \cite[Theorem 6.2]{gerhardt1983h} provides the existence of a maximal surface in $(M,g_n)$ given the existence of two smooth barriers, that is, a strictly future-convex smooth (at least $\mathcal{C}^2$) space-like surface and a strictly past-convex one. This result has been improved in \cite[Theorem 4.3]{MR2904912} reducing the regularity conditions to $\mathcal{C}^0$ barriers.

The natural candidates for these barriers are equidistant surfaces from the boundary of the convex core of $(M,g)$. It is proved in \cite[Section 5]{bonsante2009ads} that the future (respectively past) boundary component $\partial_+$ (respectively $\partial_-$) of the convex core is a future-convex (respectively past-convex) space-like pleated surface orthogonal to the particles. Moreover, each point of the boundary components is either contained in the interior of a geodesic segment (a pleating locus) or of a totally geodesic disk contained in the boundary components.

For $\epsilon>0$ fixed, consider the $2\epsilon$-surface in the future of $\partial_+$ and denote by $\partial_{+,\epsilon}$ the $\epsilon$-surface in the past of the previous one. As pointed out in \cite[Proof of Lemma 4.2]{bonsante2009ads}, this surface differs from the $\epsilon$-surface in the future of $\partial_+$ (at the pleating locus).

\begin{prop}\label{c11}
For $n$ big enough, $\partial_{+,\epsilon}\hookrightarrow (M,g_n)$ is a strictly future-convex space-like  $\mathcal{C}^{1,1}$ surface.
\end{prop}

\begin{proof}

Outside the open set $\displaystyle{U^n:=\bigcup_{j=1}^n U^n_j}$ (where the $U_j^n$ are tubular neighborhoods of $d_j$ so that the curvature is different from $-1$), $(M,g_n)$ is isometric to $(M,g)$, and  moreover, $U^n_j \underset{n\to\infty}{\longrightarrow} d_j$ for each $j$. As proved in \cite[Lemma 5.2]{bonsante2009ads}, each intersection of $\partial_+$ with a particle lies in the interior of a totally geodesic disk contained in $\partial_+$. So, there exists $n_0\in\mathbb{N}$ such that, for $n>n_0$, $U_i^j\cap \partial_+$ is totally geodesic.

The fact that $\partial_{+,\epsilon}$ is $\mathcal{C}^{1,1}$ is proved in \cite[Proof of Lemma 4.2]{bonsante2009ads}.

For the strict convexity outside $U^n$, the result is proved in \cite[Proposition 6.28]{barbot2007constant}. So it remains to prove that $\partial_{+,\epsilon}\cap U_n$ is strictly future-convex.

Let $d=d_j$ be a singular line which intersects $\partial_+$ at a point $x$. As $U:=U_n^j\cap \partial_+$ is totally geodesic, we claim that $U_\epsilon:=U_n^j \cap \partial_{+,\epsilon}$ is the $\epsilon$-surface of $U$ with respect to the metric $g_n$. In fact, the space-like surface $\mathcal{P}_0\subset \AdS^3_{\theta,i}$ given by the equation $\{t=0\}$ is totally geodesic and the one given by $\mathcal{P}_\epsilon:=\{t=\epsilon\}$ is the $\epsilon$-surface of $\mathcal{P}_0$ and corresponds to the $\epsilon$-surface in the past of $\mathcal{P}_{2\epsilon}$. It follows that $U_\epsilon$ is obtained by taking the $\epsilon$-time flow of $U$ along the unit future-pointing vector field $N$ normal to $\partial_+$ (extended to an open neighborhood of $U$ by the condition $\nabla_N^n N=0$, where $\nabla^n$ is the Levi-Civita connection of $g_n$). We are going to prove that the second fundamental form on $U_\epsilon$ is positive definite.

Note that in $\AdS^3_{\theta_j,n}$, the surfaces $\mathcal{P}_{t_0}:=\{t=t_0\}$ are equidistant from the totally geodesic space-like surface $\mathcal{P}_0$. Moreover, the induced metric on $\mathcal{P}_{t_0}$ is $\I_{t_0}=\cos^2(t_0)h_{n,\theta}$ and so, the variation of $\I_{t_0}$ along the flow of $N$ is given by 
$$\frac{d}{dt}_{\vert t=t_0} \I_t(u_t,u_t)=-2\cos(t_0)\sin(t_0),$$
for $u_t$ a unit vector field tangent to $\mathcal{P}_t$. On the other hand, this variation is given by
$$\frac{d}{dt}_{\vert t=t_0} \I_t(u_t,u_t)=\mathcal{L}_N \I_{t_0}(u_{t_0},u_{t_0}) = 2\I_{t_0}(\nabla^i_{u_{t_0}}N,u_{t_0})=-2\II_{t_0}(u_{t_0},u_{t_0}),$$
where $\mathcal{L}$ is the Lie derivative and $Bu:=-\nabla_u N$ is the shape operator.

It follows that $\II_{t_0}$ is positive-definite for $t_0>0$ small enough. So $\partial_{+,\epsilon}\hookrightarrow (M,g_n)$ is strictly future-convex.
\end{proof}

So we get a $\mathcal{C}^{1,1}$ barrier. The existence of a $\mathcal{C}^{1,1}$ strictly past-convex surface is analogous. So, by \cite[Theorem 4.3]{MR2904912}, we get that for all $n>n_0$, there exists a maximal space-like Cauchy surface $S_n$ in $(M,g_n)$. By re-indexing, we finally have proved

\begin{prop}\label{sequencemaxsurf}
There exists a sequence $(S_n)_{n\in \mathbb{N}}$ of space-like surfaces where each $S_n\hookrightarrow (M,g_n)$ is a maximal space-like surface.
\end{prop} 

\subsection{Second step}

\begin{prop}\label{limitsurf}
There exists a subsequence of $(S_n)_{n\in\mathbb{N}}$ converging uniformly on each compact which does not intersect the singular lines to a surface $S\hookrightarrow (M,g)$.
\end{prop}

\begin{proof}
For some fixed $n_0\in\mathbb{N}$, $(M,g_{n_0})$ is a smooth globally hyperbolic manifold and so admits some smooth time function $f:(M,g_{n_0})\longrightarrow \mathbb{R}$. This time function allows us to see the sequence of maximal surfaces $(S_n)_{n\in\mathbb{N}}$ as a sequence of graphs on functions over $f^{-1}(\{0\})$ (where we suppose $0\in f(M)$). Let $K\subset f^{-1}(\{0\})$ be a compact set which does not intersect the singular lines and see locally the surfaces $S_n$ as graphs of functions $u_n: K \longrightarrow \mathbb{R}$.

For $n$ big enough, the graphs of $u_n$ are pieces of space-like surfaces contained in the convex core of $(M,g)$, so the sequence $(u_n)_{n\in\mathbb{N}}$ is a sequence of uniformly bounded Lipschitz functions with uniformly bounded Lipschitz constant. By  Arzel\`a-Ascoli's Theorem, this sequence admits a subsequence (still denoted by $(u_n)_{n\in\mathbb{N}}$) converging uniformly to a function $u: K \longrightarrow \mathbb{R}$. Applying this to each compact set of $f^{-1}(\{0\})$ which does not intersect the singular line, we get that the sequence $(S_n)_{n\in\mathbb{N}}$ converges uniformly outside the singular lines to a surface $S$.
\end{proof}

Note that, as the surface $S$ is a limit of space-like surfaces, it  is nowhere time-like. However, $S$ may contains some light-like locus. We recall a theorem of C. Gerhardt \cite[Theorem 3.1]{gerhardt1983h}:

\begin{theo}(C. Gerhardt)
Let $S$ be a limit on compact subsets of a sequence of space-like surfaces in a globally hyperbolic space-time. Then if $S$ contains a segment of a null geodesic, this segment has to be
maximal, that is it extends to the boundary of $M$.
\end{theo}

So, if $S$ contains a light-like segment, either this segment extends to the boundary of $M$, or it intersects two singular lines. The first is impossible as it would imply that $S$ is not contained in the convex core. Thus, the light-like locus of $S$  lies in the set of light-like rays between two singular lines.

We now prove the following:
\begin{prop}\label{bootstrap}
The sequence of space-like surfaces $(S_n)_{n\in\mathbb{N}}$ of Proposition \ref{sequencemaxsurf} converges $\mathcal{C}^{1,1}$ on each compact which does not intersect the singular lines and light-like locus. Moreover, outside these loci, the surface $S$ has everywhere vanishing mean curvature.
\end{prop}

\begin{proof}
For a point $x\in S$ which neither lies on a singular line nor on a light-like locus, see a neighborhood $K\subset S$ of $x$ as the graph of a function $u$ over a piece of totally geodesic space-like plane $\Omega$. With an isometry $\Psi$, send $\Omega$ to the totally geodesic plane $P_0\subset \AdS^3$ given by the equation $P_0:=\{(t,\rho,\varphi)\in \AdS^3,~t=0\}$. We still denote by $S_n$ (respectively $S$, $u$ and $\Omega$) the image by $\Psi$ of $S_n$ (respectively $S$, $u$ and $\Omega$). Note that, for $n\in \mathbb{N}$ big enough, the metric $g_n$ coincides with the metric $g$ in a neighborhood of $K$ in $M$. So locally around $x$, the surfaces $S_n$ have vanishing mean curvature in $(M,g)$, hence their images in $\AdS^3$ have vanishing mean curvature.

Let $u_n:\Omega\longrightarrow \mathbb{R}$ be such that $S_n=\text{graph}(u_n)$. The unit future pointing normal vector to $S_n$ at $(x,u_n(x))$ is given by
$$N_n=v_n.\pi^*(1,\nabla u_n),$$
where $(1,\nabla u_n)\in T_x \AdS^3$ is the vector $(\partial_t,\nabla_\rho u_n,\nabla_\varphi u_n)$, $\pi: S_n \longrightarrow \Omega$ is the orthogonal projection on $P_0$ and $v_n=\big(1-\|\pi^* \nabla u_n \|^2)^{-1/2}$. The vanishing of the mean curvature of $S_n$ is equivalent to
$$-\delta_g N_n=0,$$
where $\delta_g$ is the divergence operator. In coordinates, this equation reads (see also \cite[Equation 1.14]{gerhardt1983h}):
\begin{equation}\label{max}
\frac{1}{\sqrt{\det g}}\partial_i(\sqrt{\det g}v_n g^{ij}\nabla_ju_n)+\frac{1}{2}v_n\partial_tg^{ij} \nabla_iu_n\nabla_j u_n-\frac{1}{2}v_n^{-1}g^{ij}\partial_tg_{ij}=0.
\end{equation}
Here, we wrote the metric 
$$g=-dt^2+g_{ij}(x,t)dx^idx^j,$$
applying the convention of Einstein for the summation (with indices $i,j=1,2$). The metric $g$ is taken at the points $(u_n(x),x)$ and $\det g$ is the determinant of the metric.

We have the following
\begin{lemma}
The solutions $u_n$ of equation (\ref{max}) are in $\mathcal{C}^\infty(\Omega)$.
\end{lemma}

\begin{proof}
This is a bootstrap argument. From \cite[Theorem 5.1]{gerhardt1983h}, we have $u_n\in \W^{2,p}(\Omega)$ for all $p\in[1,+\infty)$ (where $\W^{k,p}(\Omega)$ is the Sobolev space of functions over $\Omega$ admitting weak $L^p$ derivatives up to order $k$).

As $v_n$ is uniformly bounded from above and from below (because the surface $S_n$ is space-like), and as $u_n\in \W^{2,p}(\Omega)$, the third term  of equation (\ref{max}) is in $\W^{1,p}(\Omega)$. 

For the second term, we recall the multiplication law for Sobolev space: if $\frac{k}{2}-\frac{1}{p}>0$, then the product of functions in $\W^{k,p}(\Omega)$ is still in $\W^{k,p}(\Omega)$. So, as the second term of equation (\ref{max}) is a product of three terms in $\W^{1,p}(\Omega)$, it is in $\W^{1,p}(\Omega)$ (by taking $p>2$).

Hence the first term is in $\W^{1,p}(\Omega)$, and so $\sqrt{\det g}v_ng^{ij}\nabla_ju_n\in \W^{2,p}(\Omega)$. Moreover, as we can write the metric $g$ to that $g_{ij}=0$ whenever $i\neq j$ and as $\sqrt{\det g}g^{ii}$ are $\W^{2,p}(\Omega)$ and bounded from above and from below, $v_n \nabla_i u_n\in \W^{2,p}(\Omega)$. We claim that it implies $u_n\in \W^{3,p}$. It fact, for $f$ a never vanishing smooth function, consider the map
$$\begin{array}{llll}
\varphi:&  D\subset \mathbb{R}^2 & \longrightarrow & \mathbb{R}^2 \\
 & p & \longmapsto & (1-f^2(p)\vert p \vert^2)^{-1/2}p,
 \end{array}$$
where $D$ is a domain such that $f^2(p)\vert p \vert^2<1-\epsilon$ and $p\neq 0$.	The map $\varphi$ is a $\mathcal{C}^\infty$ diffeomorphism on its image, and we have $\big( \varphi(\nabla u_n)\big)_i\in \W^{2,p}(\Omega)$ for $i=1,2$  (in fact, as it is a local argument, we can always perturb $\Omega$ so that $\nabla u_n \neq 0$). Applying $\varphi^{-1}$, we get $\nabla_iu_n\in \W^{2,p}(\Omega)$ and so $u\in \W^{3,p}(\Omega)$.

Iterating the process, we obtain that $u_n \in \W^{k,p}(\Omega)$ for all $k\in \mathbb{N}$ and $p>1$ big enough. Using the Sobolev embedding Theorem 
$$\W^{j+k,p}(\Omega) \subset \mathcal{C}^{j,\alpha}(\Omega) \text{ for }0<\alpha<k-\frac{2}{p},$$
we get the result.
\end{proof}

Now, from Proposition \ref{limitsurf}, $u_n\overset{\mathcal{C}^{0,1}}{\longrightarrow} u$, that is $u_n\overset{\mathcal{W}^{1,p}}{\longrightarrow} u$ for all $p\in [1,+\infty)$.

Moreover, as the sequence of graphs of $u_n$ converges uniformly to a space-like graph, the sequence $(\nabla u_n)_{n\in \mathbb{N}}$ is uniformly bounded. From equation (\ref{max}), we get that there exists a constant $C>0$ such that for each $n\in\mathbb{N}$,
$$\vert \partial_i(\sqrt{\det h}v_ng^{ii}\nabla_iu_n)\vert <C.$$
As $(\nabla u_n)_{n\in\mathbb{N}}$ is uniformly bounded, the terms $\partial_i v_n$ are also uniformly bounded and we obtain
$$\vert \partial_i(\nabla_i u_n)\vert<C',$$
for some constant $C'$. 

Thus $(\nabla_i u_n)_{n\in\mathbb{N}}$ is a sequence of bounded Lipschitz functions with uniformly bounded Lipschitz constant so admits a convergent subsequence by Arzel\`a-Ascoli. It follows that 
$$u_n \overset{\W^{2,p}}{\longrightarrow} u,$$
for all $p\in [1,+\infty)$. Thus $u$ is a solution of equation (\ref{max}), and so $u\in\mathcal{C}^\infty(\Omega)$. Moreover, as $u$ satisfies equation (\ref{max}), $S$ has locally vanishing mean curvature.
\end{proof}

\subsection{Third step}

\begin{prop}\label{spacelike}
The surface $S$ of Proposition \ref{limitsurf} is space-like.
\end{prop}

We are going to prove that, at its intersections with the singular lines, $S$ does not contain any light-like direction. To prove this, we are going to consider the link of $S$ at its intersection $p$ with a particle $d$. The link is essentially the set of rays from $p$ that are tangent to the surface. Denote by $\alpha$ the cone angle of the singular line. We see locally the surface as the graph of a function $u$ over a small disk $D_\alpha=D_\alpha(0,r)=((0,r)\times[0,\alpha))\cup\{0\}$ contained in the totally geodesic plane orthogonal to $d$ passing through $p$ (in particular, $u(0)=0$).

First, we describe the link at a regular point of an AdS convex GHM manifold, then the link at a singular point. The link of a surface at a smooth point is a circle in a sphere with an angular metric (called \textbf{HS-surface} in \cite{MR1630178}). As the surface $S$ is a priory not smooth, we will define the link of $S$ as the domain contained between the two curves given by the limsup and liminf at zero of $\displaystyle{\frac{u(\rho,\theta)}{\rho}}$.

\medskip
\noindent\textbf{The link of a point.}
Consider $p\in (M,g)$ not lying of a singular line. The tangent space $T_pM$ identifies with the Minkowski 3-space $\mathbb{R}^{2,1}$. We define the link of $M$ at $p$, that we denote by $\mathcal{L}_p$, as the set of rays from $p$, that is the set of half-lines from $0$ in $T_p M$. Geometrically, $\mathcal{L}_p$ is a 2-sphere, and the metric is given by the angle ''distance''. So one can see that $\mathcal{L}_p$ is divided into five subsets (depending on the type of the rays and on the causality):

\begin{itemize}
\item The set of future and past pointing time-like rays that carries a hyperbolic metric.
\item The set of light-like rays defines two circles called \textbf{past and future light-like circles}.
\item The set of space-like rays which carries a de Sitter metric.
\end{itemize}

To obtain the link of a point lying on a singular line of angle $\alpha \leq 2\pi$, we cut $\mathcal{L}_p$ along two meridian separated by an angle $\alpha$ and glue by a rotation. We get a surface denoted $\mathcal{L}_{p,\alpha}$ (see Figure \ref{gluedlink}).

\begin{figure}[!h] 
\begin{center}
\includegraphics[height=7.5cm]{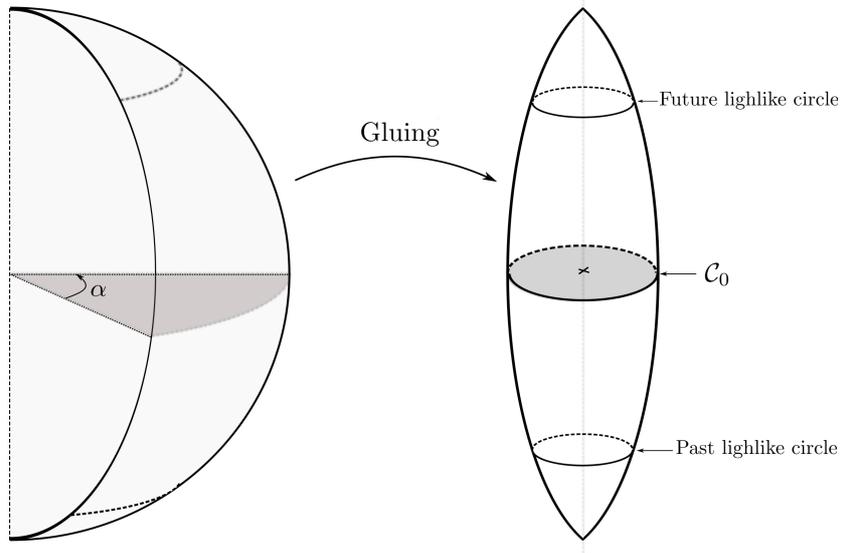}
\end{center}
\caption{Link at a singular point} 
\label{gluedlink}
\end{figure}

\medskip
\noindent\textbf{The link of a surface.}
Let $\Sigma$ be a smooth surface in $(M,g)$ and $p\in \Sigma$ not lying on a singular line. The space of rays from $p$ tangent to $\Sigma$ is just the projection of the tangent plane to $\Sigma$ on $\mathcal{L}_p$ and so describe a circle in $\mathcal{L}_p$. Denote this circle by $\mathcal{C}_{\Sigma,p}$. Obviously, if $\Sigma$ is a space-like surface, $\mathcal{C}_{\Sigma,p}$ is a space-like geodesic in the de Sitter domain of $\mathcal{L}_p$ and if $\Sigma$ is time-like or light-like, $\mathcal{C}_{\Sigma,p}$ intersects one of the time-like circle in $\mathcal{L}_p$.

Now, if $p\in \Sigma$ belongs to a singular line of angle $\alpha$ and is not smooth, we define the link of $\Sigma$ at $p$ as the domain $\mathcal{C}_{\Sigma,p}$ delimited by the limsup and the liminf of $\displaystyle{\frac{u(\rho,\theta)}{\rho}}$. 

We have an important result:

\begin{prop}
Let $\Sigma$ be a nowhere time-like surface which intersects a singular line of angle $\alpha<\pi$ at a point $p$. If $\mathcal{C}_{\Sigma,p}$  intersects a light-like circle in $\mathcal{L}_{p,\alpha}$, then $\mathcal{C}_{\Sigma,p}$ does not cross $\mathcal{C}_{0,p}$. That is, $\mathcal{C}_{\Sigma,p}$ remains strictly in one hemisphere (where a hemisphere is a connected component of $\mathcal{L}_{p,\alpha} \setminus \mathcal{C}_{0,p}$).
\end{prop}

\begin{proof}
Fix a non-zero vector $u\in T_p(\Sigma)$ and for $\theta\in [0,\alpha)$, denote by $v_\theta$ the unit vector making an angle $\theta$ with $u$. Suppose that $v_{\theta_0}$ corresponds to the direction where $\mathcal{C}_{\Sigma,p}$ intersects a light-like circle, for example, the future light-like circle. As the surface is nowhere time-like, $\Sigma$ remains in the future of the light-like plane containing $v_{\theta_0}$. But the link of a light-like plane at a non singular point $p$ is a great circle in $\mathcal{L}_p$ which intersects the two different light-like circles at the directions given by $v_{\theta_0}$ and $v_{\theta_0+\pi}$. So it intersects $\mathcal{C}_{0,p}$ at the directions $v_{\theta_0\pm\pi/2}$.

Now, if $p$ belongs to a singular line of angle $\alpha<\pi$, the link of the light-like plane which contains $v_{\theta_0}$ is obtained by cutting the link of $p$ along the directions of $v_{\theta_0\pm\alpha/2}$ and gluing the two wedges by a rotation (see the Figure \ref{gluedlink}). So, the link of our light-like plane remains in the upper hemisphere, which implies the result.
\end{proof}

\begin{rem}
Equivalently, we get that if $\mathcal{C}_{\Sigma,p}$ intersects $\mathcal{C}_{0,p}$, it does not intersect a light-like circle.
\end{rem} 

In particular, if link of $\Sigma$ at $p$, $\mathcal{C}_{\Sigma,p}$ is continuous, there exists $\eta>0$ (depending of $\alpha$) such that:
\begin{itemize}
\item If $\mathcal{C}_{\Sigma,p}$ $\theta_0$ intersects the future light-like circle, then
\begin{equation} \label{eq:1}
 u(\rho,\theta)\geq \eta.\rho~\forall\theta\in[0,\alpha),~\rho\ll 1.
 \end{equation}
\item If $\mathcal{C}_{\Sigma,p}$ $\theta_0$ intersects $\mathcal{C}_{0,p}$, then
\begin{equation} \label{eq:2}
u(\rho,\theta)\leq (1-\eta).\rho~\forall\theta\in[0,\alpha)~\rho\ll 1.
\end{equation}
\end{itemize}

\begin{figure}[!h] 
\begin{center}
\includegraphics[height=9cm]{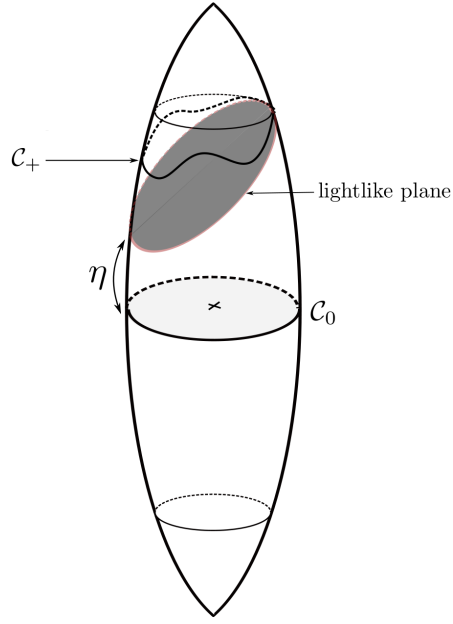}
\end{center}
\caption{The link remains in the upper hemisphere} 
\label{link2}
\end{figure}

These two results will be used in the next part.

\medskip
\noindent\textbf{Link of $S$ and orthogonality.}
Let $S$ be the limit surface of Proposition \ref{limitsurf} and let $p\in S$ be an intersection with a singular line $d$ of angle $\alpha<\pi$. As previously, we consider locally $S$ as the graph of a function
 $$u:D_\alpha \rightarrow \mathbb{R}$$
  in a neighborhood of $p$. Let $\mathcal{C}_{S,p}\subset \mathcal{L}_{p,\alpha}$ be the ``augmented'' link of $S$ at $p$, that is, the connected domain contained  between the curves $\mathcal{C}_\pm$, where $\mathcal{C}_+$ is the curve corresponding to $\displaystyle{\underset{\rho\to 0}{\limsup}\frac{u(\rho,\theta)}{\rho}}$, and $\mathcal{C}_-$ corresponding to the liminf.

\begin{lemma}
The curves $\mathcal{C}_+$ and $\mathcal{C}_-$ are $\mathcal{C}^{0,1}$.
\end{lemma}

\begin{proof}
We give the proof for $\mathcal{C}_-$ (the one for $\mathcal{C}_+$ is analogue). For $\theta\in [0,\alpha)$, denote by
$$k(\theta):=\underset{\rho\to 0}{\liminf}\frac{u(\rho,\theta)}{\rho}.$$
Fix $\theta_0,\theta\in [0,\alpha)$. By definition, there exists a decreasing sequence $(\rho_k)_{k\in\mathbb{N}}\subset\mathbb{R}_{>0}$ such that $\underset{k\to\infty}{\lim\rho_k}=0$ and
$$\underset{k\to\infty}{\lim}\frac{u(\rho_k,\theta_0)}{\rho_k}=k(\theta_0).$$
As $S$ is nowhere time-like, for each $k\in\mathbb{N}$, $S$ remains in the cone of space-like and light-like geodesic from $((\rho_k,\theta_0),u(\rho_k,\theta_0))\in S$. That is,
$$\vert u(\rho_k,\theta)-u(\rho_k,\theta_0)\vert \leq d_a(\theta,\theta_0)\rho_k,$$
where $d_a$ is the angular distance between two directions. So we get
$$\underset{k\to\infty}{\lim}\frac{u(\rho_k,\theta)}{\rho_k}\leq k(\theta_0)+d_a(\theta,\theta_0),$$
and so
$$k(\theta)\leq k(\theta_0)+d_a(\theta,\theta_0).$$
On the other hand, for all $\epsilon>0$ small enough, there exists $R>0$ such that, for all $\rho\in(0,R)$ we have:
$$u(\rho,\theta_0)>(k(\theta_0)-\epsilon)\rho.$$
By the same argument as before, because $S$ is nowhere time-like, we get
$$\vert u(\rho,\theta)-u(\rho,\theta_0)\vert \leq d_a(\theta,\theta_0)\rho,$$
that is
$$u(\rho,\theta)\geq u(\rho,\theta_0)-d_a(\theta,\theta_0)\rho.$$
So
$$u(\rho,\theta)>(k(\theta_0)-\epsilon)\rho-d_a(\theta,\theta_0)\rho,$$
taking $\epsilon\to 0$, we obtain
$$k(\theta)\geq k(\theta_0) - d_a(\theta,\theta_0).$$
So the function $k$ is 1-Lipschitz
\end{proof}

Now we can prove Proposition \ref{spacelike}. Suppose that $S$ is not space-like, that is, $S$ contains a light-like direction at an intersection with a singular line. For example, suppose that $\mathcal{C}_+$ intersects the upper light-like circle (the proof is analogue if $\mathcal{C}_-$ intersects the lower light-like circle). The proof will follow from the following lemma:

\begin{lemma} 
If the curve $\mathcal{C}_+$ intersects the future light-like circle, then $\displaystyle{\underset{\rho\to0}{\liminf}\frac{u(\rho,\theta)}{\rho}\geq\eta}$ for all $\theta \in [0,\alpha)$.
\end{lemma}

\begin{proof}

As $\mathcal{C}_+$ intersects the upper time-like circle, there exist $\theta_0\in[0,\alpha)$, and $(\rho_k)_{k\in\mathbb{N}}\subset \mathbb{R}_{>0}$ a strictly decreasing sequence, converging to zero, such that
$$\underset{k\to\infty}{\lim}\frac{u(\rho_k,\theta_0)}{\rho_k}=1.$$
From (\ref{eq:1}), for a fixed $\eta>\widetilde{\eta}$, there exist $k_0\in\mathbb{N}$ such that:
$$\forall k>k_0,~u(\rho_k,\theta)\geq\widetilde{\eta}\rho_k~\forall\theta\in[0,\alpha).$$

As $S$ has vanishing mean curvature outside its intersections with the singular locus, we can use a maximum principle. Namely, if a strictly future-convex surface $\Sigma$ intersects $S$ at a point $x$ outside the singular locus, then $S$ lies locally in the future of $\Sigma$ (the case is analogue for past-convex surfaces). It follows that on an open set $V\subset D_\alpha$, $\underset{x\in V}{\sup} u(x) = \underset{x\in \partial V}{\sup} u(x)$ and $\underset{x\in V}{\inf} u(x) = \underset{x\in \partial V}{\inf} u(x)$

Now, consider the open annulus $A_k:=D_k\setminus \overline{D}_{k+1}\subset D_\alpha$ where $D_k$ is the open disk of center 0 and radius $\rho_k$. As $S$ is a maximal surface, we can apply the maximum principle to $u$ on $A_k$, we get:
$$\inf_{A_k}u=\min_{\partial A_k}u\geq \widetilde{\eta} \rho_{k+1}.$$
So, for all $\rho\in[0,r)$, there exists $k\in \mathbb{N}$ such that $\rho\in[\rho_{k+1},\rho_k]$ and 
\begin{equation} \label{eq:3}
u(\rho,\theta)\geq\widetilde{\eta}\rho_{k+1}.
\end{equation}
We obtain that, $\forall \theta \in[0,\alpha),~u(\rho,\theta)>0$ and so $\displaystyle{\underset{\rho\to0}{\liminf}\frac{u(\rho,\theta)}{\rho}\geq 0}$. 

Now, suppose that 
$$\exists \theta_1\in[0,\alpha) \text{ such that } \displaystyle{\underset{\rho\to0}{\liminf}\frac{u(\rho,\theta_1)}{\rho}=0},$$
then there exists $(r_k)_{k\in\mathbb{N}}\subset\mathbb{R}_{>0}$ a strictly decreasing sequence converging to zero with 

$$\displaystyle{\underset{k\to\infty}{\lim}\frac{u(r_k,\theta_1)}{r_k}=0}.$$ Moreover, we can choose a subsequence of $(\rho_k)_{k\in\mathbb{N}}$ and $(r_k){k\in\mathbb{N}}$ such that $r_k\in[\rho_{k+1},\rho_k[~\forall k\in\mathbb{N}$.

This implies, by (\ref{eq:2}) that there exist $k_1\in\mathbb{N}$ such that 
$$\forall k>k_1,~ u(r_k,\theta)\leq (1-\widetilde{\eta}) r_k~\forall \theta \in [0,\alpha).$$

Now, applying the maximum principle to the open annulus $B_k:=\mathcal{D}'_k\setminus \overline{\mathcal{D}'}_{k+1}\subset D_\alpha$ where $\mathcal{D}'_k$ is the open disk of center 0 and radius $r_k$, we get:
$$\underset{B_k}{\sup u}= \underset{\partial B_k}{\max u} \leq  (1-\widetilde{\eta}) r_k.$$

And so we get that for all $\rho\in [0,r)$ there exists $k\in\mathbb{N}$ with $\rho\in [r_{k+1},r_k]$ and we have:
\begin{equation} \label{eq:4}
u(\rho,\theta)\leq (1-\widetilde{\eta})r_k\leq (1-\widetilde{\eta})\rho_k.
\end{equation}

Now we are able to prove the lemma: 

Take $\epsilon<1$, as $\displaystyle{\lim\frac{u(\rho_k,\theta_0)}{\rho_k}=1}$, there exist $k_3\in\mathbb{N}$ such that:
 
$$\forall k>k_3,~u(\rho_k,\theta_0)\geq (1-\epsilon\widetilde{\eta})\rho_k.$$
Using \eqref{eq:4} we get:

\begin{equation}\label{eq:5}
(1-\epsilon.\widetilde{\eta}).\rho_k\leq u(\rho_k,\theta_0)\leq(1-\widetilde{\eta})\rho_{k+1}, ~\text{and so }~ \frac{\rho_{k+1}}{\rho_k}\leq\frac{1-\epsilon.\widetilde{\eta}}{1-\widetilde{\eta}}.
\end{equation}

Now, as $\displaystyle{\lim\frac{u(r_k,\theta_0)}{r_k}=0}$, there exist $N'\in\mathbb{N}$ such that, for all $k$ bigger than $N'$ we have:
$$u(r_k,\theta_1)\leq\epsilon.\widetilde{\eta}.r_k\leq\epsilon.\widetilde{\eta}.\rho_k.$$
Using \eqref{eq:3}, we get:

\begin{equation}\label{eq:6}
\widetilde{\eta}.\rho_{k+1}\leq u(r_k,\theta_0)\leq\epsilon.\widetilde{\eta}.\rho_k, ~\text{and so }~ \frac{\rho_{k+1}}{\rho_k}\leq\epsilon.
\end{equation}

But as $\epsilon<1$, the conditions \eqref{eq:5} and \eqref{eq:6} are incompatibles.
\end{proof}

Now, as the curve $\mathcal{C}_-$ does not cross $\mathcal{C}_{0,p}$ and is contained in the de Sitter domain, we obtain $l(\mathcal{C}_-)<l(\mathcal{C}_{0,p})$ (where $l$ is the length). For $D_r\subset D_\alpha$ the disk of radius $r$ and center 0 and $A_g(u(D_r))$ the area of the graph of $u_{\vert D_r}$, we get:

\begin{eqnarray*}
A_g(u(D_r)) & \leq & \int_0^rl(\mathcal{C}_-)\rho d\rho \\
& < & \int_0^r l(\mathcal{C}_{0,p})\rho d\rho.
\end{eqnarray*}

The first inequality comes from the fact that $\displaystyle{\int_0^rl(\mathcal{C}_-)\rho d\rho}$ corresponds to the area of a flat piece of surface with link $\mathcal{C}_-$ which is bigger than the area of a curved surface (because we are in a Lorentzian manifold). 

So, the local deformation of $S$ sending a neighborhood of $S\cap d$ to a piece of totally geodesic disk orthogonal to the singular line would strictly increase the area of $S$. However, as $S$ is a limit of a sequence of maximal surfaces, such a deformation does not exist. So $\mathcal{C}_{S,p}$ cannot cross the light-like circles.

\subsection{Fourth step}\label{ortho}

\begin{prop}\label{fourthstep}
The surface $S\hookrightarrow (M,g)$ of Proposition \ref{limitsurf} is orthogonal to the singular lines.
\end{prop}
The proof uses a ``zooming'' argument: by a limit of a sequence of homotheties and rescaling, we send a neighborhood of an intersection of the surface $S$ with a singular line to a piece of surface $V^*$ in the Minkowski space-time with cone singularity (defined below). Then we prove, using the Gauss map, that $V^*$ is orthogonal to the singular line and we show that it implies the result. 

\begin{proof}

For $\tau>0$, define $\AdS^3_{\theta,\tau}$ as the space $(t,\rho,\varphi)\in (-\pi/2,\pi/2)\times \mathbb{R}_{\geq 0}\times [0,\theta)$ with the metric
$$g_{\theta,\tau}=-dt^2+\cos^2(t/\tau)(d\rho^2+\tau^2\sinh^2(\rho/\tau)d\theta^2).$$
Define the ``zoom'' map
$$\begin{array}{lllll}
\Z_\tau & : & \AdS^3_\theta & \longrightarrow & \AdS^3_{\theta,\tau} \\
& & (\rho,\theta,t) & \longmapsto & (\tau \rho,\tau\theta,\tau t) \\
\end{array}$$
and the set
$$K_\tau:=(K,g_{\theta,\tau}),$$
where $K:=\big\{(t,\rho,\varphi)\in  (-\pi/2,\pi/2)\times [0,1) \times [0,\theta) \big\}$.

Let $p$ be the intersection of the surface $S\hookrightarrow (M,g)$ of Proposition \ref{limitsurf} with a singular line of angle $\theta$. By definition, there exists an isometry $\Psi$ sending a neighborhood of $p$ in $M$ to a neighborhood of $0:=(0,0,0)\in \AdS^3_\theta$. Denote by $U$ the image by $\Psi$ of the neighborhood of $p$ in $S$ and set $U_n:=\Z_n(U)\cap K \subset \AdS^3_{\theta,n}$ for $n\in \mathbb{N}$. Note that the $U_n$ are pieces of space-like surface in $\AdS^3_{\theta,n}$ with vanishing mean curvature.

For all $n\in\mathbb{N}$, let $f_n: [0,1]\times \mathbb{R}/\theta\mathbb{Z}\longrightarrow [-1,1]$ so that $U_n=\text{graph}(f_n)$. With respect to the metric $d\rho^2+\sinh^2\rho d\varphi^2$ on $[0,1]\times \mathbb{R}/\theta\mathbb{Z}$, the sequence $(f_n)_{n\in\mathbb{N}}$ is a sequence of uniformly bounded Lipschitz functions with uniformly bounded Lipschitz constant and so converges $\mathcal{C}^{0,1}$ to a function $f$.

\begin{lemma}\label{blowupsurface}
Outside its intersection with the singular line, the surface $V:=\text{graph}(f)\subset K$ is space-like and has everywhere vanishing mean curvature with respect to the metric 
$$\textbf{g}_{\theta}:=-dt^2+d\rho^2+\rho^2d\theta^2.$$
\end{lemma}

\begin{proof}
As the surfaces $U_n\subset (K,g_{\theta,n})$ are space-like with everywhere vanishing mean curvature (outside the intersection with the singular line), they satisfy on $K\setminus\{0\}$ the following equation (see equation (\ref{max}), using the fact that $g_{ij}=0$ for $i\neq j$):
$$\frac{1}{\sqrt{\det g_{\theta,n}}}\partial_i(\sqrt{\det g_{\theta,n}}v_ng^{ii}_{\theta,n}\nabla_if_n)+\frac{1}{2}v_n\partial_tg^{ii}_{\theta,n}\vert \nabla_if_n\vert^2-\frac{1}{2}v_n^{-1}g^{ii}_{\theta,n}\partial_t(g_n)_{ii}=0.$$
Recall that here, $\det g_{\theta,n}$ is the determinant of the induced metric on $U_n\hookrightarrow (K,g_{\theta,n})$, $\nabla f_n$ is the gradient of $f_n$ and $v_n:=\big(1- \| \pi^* \nabla f_n\|^2\big)^{-1/2}$ for $\pi$ the orthogonal projection on $\{(t,\rho,\varphi)\in K,~t=0\}$.

As each $f_n$ satisfies the vanishing mean curvature equation, the same argument as in the proof of Proposition \ref{bootstrap} implies a uniform bound on the norm of the covariant derivative of the gradient of $f_n$. It follows that
$$f_n\overset{\mathcal{C}^{1,1}}{\longrightarrow}f.$$
Moreover, one easily checks that on $K$, $g_{\theta,n}\overset{\mathcal{C}^\infty}{\longrightarrow} \textit{\textbf{g}}_{\theta}$. In particular  $\det g_{\theta,n}$ and $v_n$ converge $\mathcal{C}^{1,1}$ to $\det \textit{\textbf{g}}_\theta$ and $v$ (respectively). It follows that $f$ is a weak solution of the vanishing mean curvature equation for the metric $\textit{\textbf{g}}_{\theta}$, and so, a bootstrap argument shows it is a strong solution. In particular, $V=\text{graph}(f)$ is a space-like surface in $(K,\textit{\textbf{g}}_{\theta})$ with everywhere vanishing mean curvature outside its intersection with the singular line.
\end{proof}

Consider on $\mathbb{R}^{2,1}$ the coordinates $(t,\rho,\varphi)\in \mathbb{R}\times \mathbb{R}_{>0}\times [0,2\pi)$ so that the metric $g$ of $\mathbb{R}^{2,1}$ writes
$$g=-dt^2+d\rho^2+\rho^2d\varphi^2.$$
The universal cover $\widetilde{\mathbb{R}^{2,1}\setminus d}$ of $\mathbb{R}^{2,1}\setminus d$ (where $d:=\{(t,\rho,\varphi)\in \mathbb{R}^{2,1},~\rho=0\}$ is the central axis) admits natural coordinates $(\widetilde{t},\widetilde{\rho},\widetilde{\varphi})\in \mathbb{R}\times \mathbb{R}_{>0}\times \mathbb{R}$. In these coordinates, the projection 
$$\pi: \widetilde{\mathbb{R}^{2,1}\setminus d} \longrightarrow \mathbb{R}^{2,1}\setminus d$$
maps $\widetilde{\varphi}$ to the unique $\varphi \in [0,2\pi)$ with $\widetilde{\varphi} \in \varphi + 2\pi\mathbb{Z}$.

Let's define 
$$r_\theta:\widetilde{\mathbb{R}^{2,1}\setminus d} \longrightarrow \widetilde{\mathbb{R}^{2,1}\setminus d}$$
by $r_\theta(\widetilde{t},\widetilde{\rho},\widetilde{\varphi})=(\widetilde{t},\widetilde{\rho},\widetilde{\varphi}+\theta)$. The quotient manifold $\mathbb{R}^{2,1}_\theta:=\widetilde{\mathbb{R}^{2,1}\setminus d} / \langle r_\theta \rangle$ inherits a singular Lorentz metric $g_\theta$ by pushing forward the metric $\pi^*g$ of $\widetilde{\mathbb{R}^{2,1}\setminus d}$. We call $\mathbb{R}^{2,1}_\theta$ with its metric the Minkowski space with cone singularity of angle $\theta$.

The manifold $(K,\textit{\textbf{g}}_\theta)$ of Lemma \ref{blowupsurface} with the central axis $l$ removed is canonically isometric to an open subset of $\mathbb{R}^{2,1}_\theta$. It follows that we can see the surface $V^*:=V\setminus\{l\cap V\}$ as a space-like surface embedded in $\mathbb{R}^{2,1}_\theta$.

Recall that the Gauss map of $V^*$ is the map associating to each point $x$ the unit future pointing vector normal to $V^*$.

\begin{lemma}
The Gauss map is naturally identified with a map $\mathcal{N}: V^* \longrightarrow \mathbb{H}^2_\theta$.
\end{lemma}

\begin{proof}
Consider $\widetilde{V^*}\subset \widetilde{\mathbb{R}^{2,1}_\theta}$ the lifting of $V^*\subset \mathbb{R}^{2,1}_\theta$. As $\widetilde{V^*}$ is space-like, for each point $p\in \widetilde{V^*}$, the geodesic orthogonal to $\widetilde{V^*}$ passing through $p$ either intersects the space-like surface $\widetilde{\mathbb{H}^{2*}}:=\{(\widetilde{t},\widetilde{\rho},\widetilde{\varphi})\in \widetilde{\mathbb{R}^{2,1}_\theta},~\widetilde{t}=\sqrt{\widetilde{\rho}^2+1}\}$ (which is a lifting of the hyperboloid $\mathbb{H}^2\subset \mathbb{R}^{2,1}$ with the point $(1,0,0)$ removed) or is not complete (namely, the geodesic hits the boundary curve $\{\widetilde{\rho}=0\}$).

Denote by $\widetilde{\frak{p}}\subset \widetilde{V^*}$ the set of points so that the orthogonal geodesic is not complete. The Gauss map is thus canonically identified with a map
$$\widetilde{\mathcal{N}}: \widetilde{V*}\setminus \widetilde{\frak{p}} \longrightarrow \widetilde{\mathbb{H}^{2*}}.$$
It is clear that $\widetilde{\mathcal{N}}$ is equivariant with respect to the action of $r_\theta$ so descends to a map
$$\mathcal{N}: V^*\setminus \frak{p} \longrightarrow \mathbb{H}^{2*}_\theta:= \widetilde{\mathbb{H}^{2*}} / \langle r_\theta \rangle,$$
where $\frak{p}=\widetilde{\frak{p}}/ \langle r_\theta \rangle$. Note that $\mathbb{H}^{2*}_\theta$ is isometric to the hyperbolic disk with cone singularity (defined in the Introduction) with the center $0_\theta$ removed.

As $V^*$ is smooth, setting $\mathcal{N}(\frak{p})=0_\theta$ gives a smooth extension of $\mathcal{N}$ to a map $\mathcal{N}: V^* \longrightarrow \mathbb{H}^2_\theta.$
\end{proof}

\begin{lemma}
The Gauss map $\mathcal{N}: V^* \longrightarrow \mathbb{H}^2_\theta$ is holomorphic with respect to the complex structure associated to the reverse orientation of $\mathbb{H}^2_\theta$.
\end{lemma}

\begin{proof}
As $V^*$ has everywhere vanishing mean curvature, we can choose an orthonormal framing of $TV^*$ such that the shape operator $B$ of $V^*$ expresses
$$B=\begin{pmatrix}
k & 0 \\
0 & -k
\end{pmatrix}.$$
Denoting $h_\theta$ the metric of $\mathbb{H}^2_\theta$, we obtain that
$$\mathcal{N}^*h_\theta=\I(B.,B.)=k^2\I(.,.),$$
where $\I$ is the first fundamental form of $V^*$. That is $\mathcal{N}$ is conformal and reverses the orientation and so is holomorphic with respect to the holomorphic structure defined by the opposite orientation of $\mathbb{H}^2_\theta$.
\end{proof}

\begin{lemma}
The piece of surface $V^* \hookrightarrow \mathbb{R}^{2,1}_\theta$ is orthogonal to the singular line.
\end{lemma}

\begin{proof}
Fix complex coordinates $z: V^* \longrightarrow \mathbb{D}^*$ and $w: \mathbb{H}^2_\theta \longrightarrow \mathbb{D}^*$. In these coordinates systems, the metric $g_V$ and $h_\theta$ of $V^*$ and $\mathbb{H}^2_\theta$ respectively express:
$$g_V= \rho^2(z)\vert dz\vert^2,~~ h_\theta= \sigma^2(w)\vert dw\vert^2.$$
Note that, as $\mathbb{H}^2_\theta$ carries a conical singularity of angle $\theta$ at the center, $\sigma^2(w)=e^{2u}\vert w\vert^{2(\theta/2\pi-1)}$, where $u$ is a bounded function.

Assuming $\mathcal{N}$ does not have an essential singularity at $0$, the expression of $\mathcal{N}$ in the complex charts has the form:
$$\mathcal{N}(z)= \frac{\lambda}{z^n} + f(z),\text{ where }z^nf(z)\underset{z\to 0}{\longrightarrow} 0$$
for some $n\in \mathbb{Z}$ and non-zero $\lambda$.

Denote by $e(\mathcal{N})=\frac{1}{2}\Vert d\mathcal{N}\Vert^2$ the energy density of $\mathcal{N}$. The third fundamental form of $V^*$ is thus given by
$$\mathcal{N}^*h_\theta= e(\mathcal{N})g_U.$$
Moreover, we have:
$$e(\mathcal{N})= \rho^{-2}(z)\sigma^2(\mathcal{N}(z)) \vert \partial_z \mathcal{N}\vert^2.$$
If $n\neq 0$, we have
$$\vert \partial_z \mathcal{N}\vert^2 = C \vert z\vert^{2(n-1)}+ o\left(\vert z\vert^{2(n-1)}\right),\text{ for some }C>0,$$
and 
$$\sigma^2(\mathcal{N}(z))=e^{2v}\vert z\vert^{2n(\theta/2\pi-1)},\text{ for some bounded } v.$$
So we finally get,
$$ \mathcal{N}^* h_\theta=e^{2\varphi}\vert z\vert^{2(n\theta/2\pi -1)} \vert dz\vert^2,\text{ where }\varphi \text{ is bounded.}$$
For $n=0$, the same computation gives
$$\mathcal{N}^* h_\theta= e^{2\varphi} \vert dz\vert^2,\text{ where }\varphi \text{ is bounded from above.}$$
For $\mathcal{N}$ having an essential singularity, we get that for all $n<0,~\vert z\vert^n=o\left(\rho^2(z)e(\mathcal{N})\right)$ and so $\mathcal{N}^*h_\theta$ cannot have a conical singularity.

However, as the third fundamental form is the pull-back by the Gauss map of $h_\theta$, it has to carry a conical singularity of angle $\theta$, that is have the expression
$$\III=e^{2\Psi}\vert z\vert^{2(\theta/2\pi-1)} \vert dz\vert^2,$$
for some bounded $\Psi$. It implies in particular that $n=1$, and so $\mathcal{N}(z)\underset{z\to 0}{\longrightarrow} 0$, which means that $V^*$ is orthogonal to the singular line.

\end{proof}

The proof of Proposition \ref{fourthstep} follows. For $\tau\in \mathbb{R}_{>0}$, let $u_\tau\in T_0 \AdS^3_{\theta,\tau}$ be the unit future pointing vector tangent to $d$ at $0=U_\tau\cap d$. For $x\in U_\tau$ close enough to $0$, let $u_\tau(x)$ be the parallel transport of $u_\tau$ along the unique geodesic in $U_\tau$ joining $0$ to $x$. Denoting by $\N_\tau$ the Gauss map of $U_\tau$, we define a map:
$$\psi_\tau(x):=g_{\theta,\tau}(u_\tau(x),\N_\tau(x)),$$
where $g_{\theta,\tau}$ is the metric of $\AdS^3_{\theta,\tau}$. Note that, by construction, the value of $\psi_\tau(0)$ is constant for all $\tau\in\mathbb{R}_{>0}$. As $U_\infty$ is orthogonal to $d$, $\underset{\tau\to\infty}{\lim} \psi_\tau(0)=-1$ so in particular $\psi_1(0)=-1$, that is the surface $S$ is orthogonal to the singular lines.
\end{proof}

\section{Uniqueness}

In this section, we prove the uniqueness part of Theorem \ref{surf}:

\begin{prop}\label{uniq}
The maximal surface $S\hookrightarrow (M,g)$ of Proposition \ref{existence} is unique.
\end{prop}

\begin{proof}

Given a causal curve intersecting two space-like surfaces $S_1$ and $S_2$, we denote by $l_\gamma (S_1,S_2)$ the causal length of $\gamma$ between $S_1$ and $S_2$. Namely,
$$l_\gamma (S_1,S_2)=\int_{t_1}^{t_2} \big(-g(\gamma'(t),\gamma'(t))\big)^{1/2}dt,$$
where $\gamma(t_i)\in S_i,~i=1,2$.

Suppose that there exist two different maximal surfaces $S_1 \text{ and } S_2$ in $(M,g)$ where $S_1$ is the one of Proposition \ref{existence}. Denote by $\Gamma$ the set of time-like geodesics in $M$ and set
$$C:=\sup_{\gamma \in \Gamma} \{l_\gamma(S_1,S_2)\}>0.$$

Note that, from \cite[Lemma 5.7]{bonsante2009ads}, as $S_1\hookrightarrow (M,g)$ is contained in the convex core, $C<\pi/2$. Consider $(\gamma_n)_{n\in \mathbb{N}}\subset\Gamma$ such that 
$$\underset{n\to\infty}{\lim} l(\gamma_n)=C.$$

\begin{lemma}
The sequence of geodesic segments $(\gamma_n)_{n\in\mathbb{N}}$ converges to $\gamma\in\Gamma$.
\end{lemma}
\begin{proof}
For $i=1,2$, denote by $x_{in}$ the intersection of $\gamma_n$ and $S_i$.

For $n\in\mathbb{N}$, choose a lifting $\widetilde{x_1}_n$ of $x_{1n}$ in the universal cover $\widetilde{M}$ of $M$. This choice fixes a lifting of the whole sequence $(x_{1n})_{n\in\mathbb{N}}$ and of  $(\gamma_n)_{n\in\mathbb{N}}$, so of $(x_{2n})_{n\in\mathbb{N}}$.  Note that the sequence $(\widetilde{x_1}_n)_{n\in\mathbb{N}}$ converges to $\widetilde{x_1}\in \widetilde{S_1} \subset \widetilde{M}$ and, as the causal cone of $\widetilde{x_1}$ intersects $\widetilde{S}_2$ in a compact set containing infinitely many $\widetilde{x_2}_n$, the sequence $(\widetilde{x_2}_n)_{n\in\mathbb{N}}$ converges to $\widetilde{x}_2$ (up to a subsequence).

It follows that $\widetilde{x_2}$ projects to $x_2 \in S_2$ and $C=l_\gamma(S_1,S_2)$ where $\gamma$ is the projection of the unique time-like geodesic joining $\widetilde{x}_1$ to $\widetilde{x}_2$.
\end{proof}

It is clear that $S_1$ and $S_2$ are orthogonal to $\gamma$ (if not, there would exist some deformation of $\gamma$ increasing the causal length at the first order).

For $i=1,2$, denote by $x_i$ the intersection of $\gamma$ and $S_i$, by $P_i$ the (locally defined) totally geodesic plane tangent to $S_i$ at $x_i$ and by $\pm k_i$ the principal curvatures of $S_i$ at $x_i$. We can assume moreover that $k_1\geq k_2 \geq 0$ and that $x_2$ is in the future of $x_1$.

Let $u\in \mathcal{U}_{x_1} S_1$ (where $\mathcal{U}_x S_1$ is the unit tangent bundle to $S_1$) be a principal direction associated to $k_1$. For $\epsilon>0$, denote by $J$ the Jacobi field along $\gamma$ so that 
$$\left\{ \begin{array}{lll}
J(0) & = & \epsilon u \\
J'(0) & = & 0
\end{array}\right.$$
and set $\gamma_\epsilon:= \exp(J) \in \Gamma$ the deformation of $\gamma$ along $J$. It is clear that $\gamma_\epsilon$ is orthogonal to the piece of totally geodesic plane $P_1$.

By definition of the curvature, we have
$$l_{\gamma_\epsilon}(S_1,S_2) =  l_{\gamma_\epsilon}(P_1,P_2) + (k_1-\kappa_2)\epsilon^2 +o(\epsilon^2).$$
Here $\kappa_2$ is the curvature of $S_2$ at $x_2$ is the direction $p_\gamma(u)$ where $p_\gamma : T_{x_1} S_1 \longrightarrow T_{x_2} S_2$ is the parallel transport along $\gamma$. In particular we have $-\kappa_2\geq -k_2$ and so
$$l_{\gamma_\epsilon}(S_1,S_2) \geq  l_{\gamma_\epsilon}(P_1,P_2).$$

Moreover, in the proof of Proposition \ref{c11} we proved that the equidistant surface in the future of a totally geodesic space-like plane in $\AdS^3_\theta$ is strictly future-convex (when the distance is less than $\pi/2$). It follows that
$$ l_{\gamma_\epsilon}(P_1,P_2)> C,$$
and we get a contradiction.
\end{proof}

\section{Consequences}

\subsection{Minimal Lagrangian diffeomorphisms}

In this paragraph, we prove Theorem \ref{min}. Let $\Sigma$ be a closed oriented surface endowed with a Riemannian metric $g$ and let $\nabla$ be the associated Levi-Civita connection.

\begin{Def}
A bundle morphism $b: T\Sigma\longrightarrow T\Sigma$ is \textbf{Codazzi} if $d^\nabla b=0$, where $d^\nabla$ is the covariant derivative of vector valued form associated to the connection $\nabla$.
\end{Def}

We recall a result of \cite{labourie1992surfaces}:
\begin{theo}[Labourie]\label{labourie}
Let $b:T\Sigma\longrightarrow T\Sigma$ be a everywhere invertible Codazzi bundle morphism, and let $h$ be the symmetric $2$-tensor defined by $h=g(b.,b.)$. The Levi-Civita connection $\nabla^h$ of $h$ satisfies
$$\nabla^h_uv=b^{-1}\nabla_u(bv),$$
and its curvature is given by:
$$K_h=\frac{K_g}{\det(b)}.$$
\end{theo}

Given $g_1,g_2 \in \T$ and $\Psi: (\Sigma_\fkp,g_1)\longrightarrow (\Sigma_\fkp,g_2)$ a diffeomorphism isotopic to the identity, there exists a unique bundle morphism $b: T\Sigma_\fkp \longrightarrow T\Sigma_\fkp$ so that $g_2=g_1(b.,b.)$. We have the following characterization:

\begin{prop}\label{minlag}
The diffeomorphism $\Psi$ is minimal Lagrangian if and only if
\begin{itemize}
\item[i.] $b$ is Codazzi with respect to $g_1$,
\item[ii.] $b$ is self-adjoint for $g_1$ with positive eigenvalues.
\item[iii.] $\det(b)=1$.
\end{itemize}
\end{prop}

We now prove Theorem \ref{min}:

\textbf{Existence:}
Let $g_1,g_2\in \T$. It follows from Theorem \ref{surf} and the extension of Mess' parametrization that we can uniquely realize $g_1$ and $g_2$ as
 $$\left\{
 \begin{array}{l}
 g_1=  \I((E+JB).,(E+JB).) \\
 g_2= \I((E-JB).,(E-JB).),
 \end{array}\right.$$
 where $\I,~B,~J,~E$ are respectively the first fundamental form, the shape operator, the complex structure and the identity morphism of the unique maximal surface $S\hookrightarrow (M,g)$ and $(M,g)$ is an AdS convex GHM space-time with particles parametrized by $(g_1,g_2)$.
 
Define the bundle morphism $b:T\Sigma_\fkp\longrightarrow T\Sigma_\fkp$:
$$b=(E+JB)^{-1}(E-JB).$$
Note that, as the eigenvalues of $B$ are in $(-1,1)$, (from \cite[Lemma 5.15]{krasnov2007minimal}) the morphism $b$ is well defined. Moreover, we have $g_2=g_1(b.,b.)$. We are going to prove that $b$ satisfies the conditions of Proposition \ref{minlag}:
\begin{enumerate}
\item[-]\textit{Codazzi:} Denote by $D$ the Levi-Civita connection associated to $\I$, and consider the bundle morphism $A=(E+JB)$. From Codazzi's equation for surfaces, $d^DA=0$. From Proposition \ref{labourie}, the Levi-Civita connection $\nabla_1$ of $\I(A.,A.)$ satisfies:
$$\nabla_{1u}v=A^{-1}D_u(Av).$$
We get that $d^{\nabla_1} b=A^{-1}d^D(E-JB)=0$.
\item[-]\textit{Self-adjoint:}
\begin{eqnarray*}
g_1(bx,y) & =  & \I\big((E-JB)x,(E+JB)y\big) \\
& = & \I\big((E+JB)(E-JB)x,y\big) \\
& =  & \I\big((E-JB)(E+JB)x,y\big) \\
& = & \I\big((E+JB)x,(E-JB)y\big) \\
& = & g_1(x,by).
\end{eqnarray*}
\item[-]\textit{Positive eigenvalues:}
From \cite[Lemma 5.15]{krasnov2007minimal}, the eigenvalues of $B$ are in $(-1,1)$. So $(E\pm JB)$ has strictly positives eigenvalues and the same hold for $b$.
\item[-]\textit{Determinant 1:} $\displaystyle{\det(b)=\frac{\det(E-JB)}{\det(E+JB)}=\frac{1+\det(JB)}{1+\det(JB)}=1}$, (because $\text{tr}(JB)=0$).
\end{enumerate}

\textbf{Uniqueness:}
Suppose that there exist $\Psi_1,\Psi_2: (\Sigma_\fkp,g_1)\longrightarrow (\Sigma_\fkp,g_2)$ two minimal Lagrangian diffeomorphisms. It follows from Proposition \ref{minlag} that there exists $b_1,b_2: T\Sigma_\fkp \longrightarrow T\Sigma_\fkp$ Codazzi self-adjoint with respect to $g_1$ with positive eigenvalues and determinant 1 so that $g_1(b_1.,b_1.)$ and $g_2(b_2.,b_2.)$ are in the same isotopy class.

For $i=1,2$, define
$$\left\{\begin{array}{ll}
\I_i(.,.)&=  \frac{1}{4} g_1\big((E+b_i).,(E+b_i).\big) \\
B_i&=  -J_i(E+b_i)^{-1}(E-b_i), \\
\end{array}\right.$$
where $J_i$ is the complex structure associated to $\I_i$.

One easily checks that $B_i$ is well defined and self-adjoint with respect to $\I_i$ with eigenvalues in $(-1,1)$. Moreover, we have
$$b_i  =  (E+J_iB_i)^{-1}(E-J_iB_i).$$

Writing the Levi-Civita connection of $g_1$ by $\nabla$ and the one of $\I_i$ by $D^i$, Proposition \ref{labourie} implies

$$D^i_xy=(E+b_i)^{-1}\nabla_x((E+b_i)y).$$
So we get:
\begin{align*}
D^iB_i(x,y)&=(E+b_i)^{-1}\nabla_y\big((E+b_i)By\big)-(E+b_i)^{-1}\nabla_y\big((E+b_i)x\big)-B_i[x,y] \\
& =  (E+b_i)^{-1}(\nabla(E+b_i))(x,y) \\
& =  0.
\end{align*}
And the curvature of $\I_i$ satisfies
$$K_{\I_i}=-\det(E+JB_i)=-1-\det(B_i).$$
It follows that $B_i$ is traceless, self-adjoint and satisfies the Codazzi and Gauss equation. Setting $\II_i:=\I_i(B_i.,.)$, we get that $\I_i$ and $\II_i$ are respectively the first and second fundamental form of a maximal surface in an AdS convex GHM manifold with particles (that is, $(\I_i,\II_i)\in \HH$ where $\HH$ is defined in Section \ref{adsparticles}). Moreover, one easily checks that, for $i=1,2$
$$\left\{\begin{array}{l}
g_1  =  \I_i\left( (E+J_iB_i).,(E+J_iB_i).\right) \\
g_2  =  \I_i \left( (E-J_iB_i).,(E-J_iB_i).\right) \\
\end{array}\right.$$

It means that  $(\I_i,\II_i)$ is the first and second fundamental form of a maximal surface in $(M,g)$ (for $i=1,2$) and so, by uniqueness, $(\I_1,\II_1)=(\I_2,\II_2)$. In particular, $b_1=b_2$ and $\Psi_1=\Psi_2$.

\subsection{Middle point in $\T$}\label{interpretation}

Theorem \ref{min} provides a canonical identification between the moduli space $\M$ of singular AdS convex GHM structure on $\Sigma_\fkp\times \mathbb{R}$ with the space $\HH$ of maximal surfaces in germ of singular AdS convex GHM structure (as defined in Section \ref{adsparticles}). By the extension of Mess' parametrization, the moduli space $\M$ is parametrized by $\T\times\T$ and by \cite[Theorem 5.11]{krasnov2007minimal}, the space $\HH$ is parametrized by $T^*\T$.

It follows that we get a map
$$\varphi: \T\times \T \longrightarrow T^*\T.$$
We show that this map gives a ``middle point'' in $\T$:

\begin{theo}\label{interpretation}
Let $g_1,g_2\in \T$ be two hyperbolic metrics with cone singularities. There exists a unique conformal structure $\frak{c}$ on $\Sigma_\fkp$ so that
$$\Phi(u_1)=-\Phi(u_2)$$
and $u_2\circ u_1^{-1}$ is minimal Lagrangian. Here $u_i: (\Sigma_\fkp,\frak{c})\longrightarrow (\Sigma_\fkp,g_i)$ is the unique harmonic map isotopic to the identity and $\Phi(u_i)$ is its Hopf differential. Moreover, 
$$(g_1,g_2)=\varphi(\frak{c},i\Phi(u_1)).$$
\end{theo}

\begin{proof}
\noindent\textbf{Existence:}
For $g_1,g_2\in\T$, let $\Psi: (\Sigma_\fkp,g_1) \longrightarrow (\Sigma_\fkp,g_2)$ be the unique minimal Lagrangian diffeomorphism isotopic to the identity. By definition, the embedding $\iota:~\Gamma:=\text{graph}(\Psi) \hookrightarrow (\Sigma_\fkp\times \Sigma_\fkp, g_1\oplus g_2)$ is minimal so $\iota$ is a conformal harmonic map (see \cite[Proposition 4.B]{eells-sampson}).  

It follows that the $i$-th projection $\pi_i :(\Sigma_\fkp\times \Sigma_\fkp, g_1\oplus g_2) \longrightarrow (\Sigma_\fkp,g_i)$ restricts to a harmonic map $u_i$ on $\iota(\Gamma)$. Moreover, as $\iota$ is conformal, we get
$$\iota^*(g_1\oplus g_2)^{2,0}=0=u_1^*(g_1)^{2,0}+u_2^*(g_2)^{2,0}=\Phi(u_1)+\Phi(u_2).$$

\medskip
\noindent\textbf{Uniqueness:} It is a direct consequence of the uniqueness part of Theorem \ref{min}.

\medskip
\noindent\textbf{Expression of $\varphi$:}
It follows from the extension of Mess' parametrization (section \ref{extensionmessparam}) and Theorem \ref{surf} that the metrics $g_1$ and $g_2$ can be uniquely expressed as

$$\left \{
\begin{array}{l}
g_1=\I((E+JB).,(E+JB).) \\
g_2=\I((E-JB).,(E-JB).),   
\end{array}
\right .$$

\noindent where $\I,~B,~J,~E$ are respectively the first fundamental form, shape operator, complex structure and identity morphism of the unique maximal surface $S\hookrightarrow (M,g)$, and $(M,g)$ is an AdS convex GHM space-time with particles parametrized by $(g_1,g_2)$.

An easy computation shows that
$$g_1+g_2=2(\I+\III).$$

As $S$ is maximal, the third fundamental form $\III$ of $S$ is conformal to $\I$, so
$$\frak{c}:=[g_1+g_2]=[I],$$
where $[h]$ denotes the conformal class of a metric $h$.

On the other hand, the embedding $\iota: \Gamma=\text{graph}(\Psi) \hookrightarrow (\Sigma_\fkp\times \Sigma_\fkp, g_1\oplus g_2)$ is conformal, so
$$[\iota^*(g_1\oplus g_2)]=[g_1+g_2]=[I]=\frak{c}.$$

By definition of the Hopf differential, there exists some strictly positive function $\lambda$ so that with respect to the complex structure associated to $\frak{c}$, we have the following decomposition
$$u_1^*g_1= \Phi(u_1) + \lambda \iota^*(g_1\oplus g_2)+ \overline{\Phi(u_1)}.$$
Note that in our case, the metrics $g_1$ and $g_2$ are normalized so that $u_i=\text{id}$. Moreover, as $[\iota^*g_1\oplus g_2]=[\I]$, we have (for a different function $\lambda'$)
$$g_1= \Phi(u_1) + \lambda' \I+ \overline{\Phi(u_1)}.$$
Now we get
$$g_1=\I\big((E+JB).,(E+JB).\big)= 2\I(JB.,.)+ \I+\III,$$
in particular,
$$\I(JB.,.)=\Re\big(\Phi(u_1)\big).$$

Let $\partial_x,\partial_y\in \Gamma(TS)$ be an orthonormal framing of principal directions of $S$. We have the following expressions in this framing:
$$B=
\left(\begin{array}{ll}
k & 0 \\
0 & -k
\end{array}\right),~~
J=
\left(\begin{array}{ll}
0 & -1 \\
1 & 0
\end{array}\right).$$
Setting $dz\in \Gamma(T^*S\otimes\mathbb{C}),~dz:=dx+idy$ where $(dx,dy)$ is the framing dual to $(\partial_x,\partial_y)$, we obtain
$$\Phi(u_1)= -ik dz^2=k(dxdy+dydx)-ik(dx^2-dy^2).$$
So
$$\II=\I(B.,.)=\Re(i\Phi(u_1)),$$
and $(\frak{c},i\Phi(u_1))$ is the maximal AdS germ associated to $(M,g)$.
\end{proof}

\bibliographystyle{alpha}
\bibliography{Maximal Surface in AdS with particles.bbl}

\end{document}